\documentclass[11pt,oneside,english,reqno]{amsart}
\usepackage[T1]{fontenc}
\usepackage[latin9]{inputenc}
\usepackage[a4paper]{geometry}
\usepackage{mathrsfs}
\usepackage{amstext}
\usepackage{amsthm}
\usepackage{amssymb}
\usepackage{color}
\usepackage{hyperref}
\usepackage{dsfont}
\usepackage{enumerate}
\usepackage{verbatim} 
\usepackage{stmaryrd}
\usepackage{enumitem}
\usepackage[dvipsnames]{xcolor}

\makeatletter
\numberwithin{equation}{section}
\numberwithin{figure}{section}
\theoremstyle{plain}
\newtheorem{thm}{\protect\theoremname}
\theoremstyle{definition}
\newtheorem{defn}[thm]{\protect\definitionname}
\theoremstyle{remark}
\newtheorem{rem}[thm]{\protect\remarkname}
\theoremstyle{plain}
\newtheorem{lem}[thm]{\protect\lemmaname}
\theoremstyle{plain}
\newtheorem{prop}[thm]{\protect\propositionname}
\theoremstyle{plain}
\newtheorem{cor}[thm]{\protect\corollaryname}
\theoremstyle{plain}
\newtheorem{ex}[thm]{\protect\examplename}
\theoremstyle{plain}
\newtheorem{ass}[thm]{\protect\assumptionname}
\numberwithin{thm}{section}
\makeatother

\usepackage{babel}
\providecommand{\corollaryname}{Corollary}
\providecommand{\definitionname}{Definition}
\providecommand{\lemmaname}{Lemma}
\providecommand{\propositionname}{Proposition}
\providecommand{\remarkname}{Remark}
\providecommand{\theoremname}{Theorem}
\providecommand{\examplename}{Example}
\providecommand{\assumptionname}{Assumption}

\newcommand{\red}[1]{{\color{red} #1}}

\newcommand{\cB}{\mathcal{B}}

\newcommand{\cE}{\mathcal{E}}
\newcommand{\cF}{\mathcal{F}}

\newcommand{\cI}{\mathcal{I}}

\newcommand{\cL}{\mathcal{L}}

\newcommand{\cP}{\mathcal{P}}

\newcommand{\EE}{\mathbb{E}}

\newcommand{\NN}{\mathbb{N}}

\newcommand{\PP}{\mathbb{P}}

\newcommand{\RR}{\mathbb{R}}

\newcommand{\mathd}{\mathrm{d}}

\newcommand{\vertiii}[1]{{\left\vert\kern-0.25ex\left\vert\kern-0.25ex\left\vert #1 
  \right\vert\kern-0.25ex\right\vert\kern-0.25ex\right\vert}}

\newcommand{\eps}{\varepsilon}
\newcommand{\dd}{\mathop{}\!\mathrm{d}}

\let\div\relax
\DeclareMathOperator{\div}{div}

\long\def\avi#1{{\color{red}Avi:\ #1}}
\long\def\lucio#1{{\color{blue}Lucio:\ #1}}

 \hypersetup{
   colorlinks  = true,
   citecolor  = blue,
   linkcolor  = blue
}

\title{Distribution dependent SDEs driven by Additive Continuous Noise}
\author{ Lucio Galeati \and Fabian A. Harang \and Avi Mayorcas}
\date{\today}
\address{Lucio Galeati: 
Institute of Applied Mathematics, University of Bonn, 53115 Endenicher Allee 60, Bonn, Germany.
}
\email{lucio.galeati@iam.uni-bonn.de}
\address{Fabian A. Harang: 
 Department of Mathematics, University of Oslo, P.O. box 1053, Blindern, 0316, OSLO, Norway.}
\email{fabianah@math.uio.no} 
\address{Avi J. Mayorcas: Mathematical Institute, University of Oxford, Oxford, OX2 6GG, UK.}
\email{mayorcas@math.ox.ac.uk}

\begin{document}

\begin{abstract}
  We study distribution dependent stochastic differential equation driven by a continuous process, without any specification on its law, following the approach initiated in \cite{coghi2020pathwise}. We provide several criteria for existence and uniqueness of solutions which go beyond the classical globally Lipschitz setting.
  In particular we show well-posedness of the equation, as well as almost sure convergence of the associated particle system, for drifts satisfying either Osgood-continuity, monotonicity, local Lipschitz or Sobolev differentiability type assumptions.
\end{abstract}

\thanks{
{\em Acknowledgments:} FH gratefully acknowledges financial support from the STORM project 274410, funded by the Research Council of Norway.
LG is funded by the DFG under Germany's Excellence Strategy - GZ 2047/1, project-id 390685813.}

\maketitle
{
\hypersetup{linkcolor=black}
 \tableofcontents 
}

\section{Introduction}
In this work we consider distribution dependent SDEs (henceforth DDSDEs) of the form
\begin{equation}\label{eq:introddsde}
  X_t = \int_0^t B_s(X_s,\mu_s)\dd s + Y_t,\quad \mu_t=\cL(X_t)
\end{equation}
where $B : \RR_+\times \RR^d \times \cP(\RR^d)\rightarrow \RR^d$, $X$ is the unknown and $Y$ is a continuous stochastic process with prescribed law; $\cL(X_t)$ stands for law of the random variable $X_t$.

In the literature the process $Y$ is typically sampled as $Y_t =\xi + \sqrt{2\eps}\, W_t$, where $\eps\geq 0$, $W$ is a standard Brownian motion and $\xi$ is an $\RR^d$-valued random variable independent of $W$, encoding the initial condition. In this case the DDSDE is also called a McKean--Vlasov SDE, after the pioneering work \cite{mckean1966class} where it was first introduced (setting $\eps=0$ corresponds to the Vlasov equation).

One reason for the importance of McKean--Vlasov equations is their connection to systems of $N$ particles subject to a mean field interaction of the form
\begin{equation}\label{eq:intro-particle-system}
  X^{i,N}_t = \int_0^t B_s\big(X^{i,N}_s, L^N(X^{(N)}_s)\big) \dd s + Y^{i}_t, \quad L^N(X^{(N)}_t):=\frac{1}{N}\sum_{i=1}^N \delta_{X^{i,N}_t}
\end{equation}
for all $i=1,\ldots, N$; where $L^N(X^{(N)}_t)$ stands for the empirical measure of the system at time $t$ and $Y^i$ are usually sampled as i.i.d. copies of $Y$.
One expects the DDSDE \eqref{eq:introddsde} to be the mean field limit of \eqref{eq:intro-particle-system} in the sense that, as $N$ goes to infinity, $L^N(X^{(N)}_t)$ converges weakly to $\cL(X_t)$ with probability $1$.
Moreover for $Y=\xi + \sqrt{2\eps}\, W$, by It\^o calculus, the evolution of the marginal $\cL(X_t)$ is given by the nonlinear Fokker--Planck equation (also called McKean--Vlasov equation)
\begin{equation}\label{eq:intro-PDE}
  \partial_t \rho + \nabla \cdot ((B_t(\,\cdot\,, \rho)\, \rho) = \eps \Delta \rho, \quad \rho_0=\cL(\xi).
\end{equation}
In particular, both \eqref{eq:introddsde} and \eqref{eq:intro-PDE} provide a macroscopic, compact description of system \eqref{eq:intro-particle-system}, reducing its complexity.
For this reason, DDSDEs have found applications in numerous fields, including galaxy formation, plasma physics, $2$D fluid dynamics, bacteria chemotaxis, agent-based modelling, neuroscience, flocking and swarming dynamics; we refer the interested reader to \cite{coghi2020pathwise}, the review \cite{jabin2017mean} and the references therein. In recent years they have also attracted a lot of attention due to their connection to mean-field games \cite{lasry2007mean}.

Classical results concerning the mean-field limit property go back to Dobrushin \cite{dobrushin1979vlasov} for $\eps=0$ and to Sznitman \cite{sznitman1991topics} for general $\eps \geq 0$, under Lipschitz continuity of $B$; let us also mention the work by G\"artner \cite{Gartner1988} which remains among the most general on the topic.

Recently the field has received a lot of attention both from the analytic and probabilistic communities.
On one hand, new methods based on entropy inequalities \cite{fournier2014propagation, jabin2018quantitative, bresch2019mean} and modulated energy \cite{serfaty2017mean, serfaty2020mean} have allowed for the rigorous derivation of mean field limits for fairly singular $B$;
on the other, DDSDEs with very irregular drifts are related to the flourishing field of regularization by noise phenomena, see \cite{BauerMBrandisProske2018,mishura2016existence,rockner2018well,chaudru2020strong,lacker2018strong}.

Most of the works mentioned above rely on exploiting either the connection to the PDE \eqref{eq:intro-PDE} or stochastic analysis tools related to the Brownian motion $W$.
In this work we instead focus on studying the DDSDE \eqref{eq:introddsde} with as minimal assumptions on the noise $Y$ as possible, usually merely continuity and finiteness of moments. 
We are strongly inspired by the recent work \cite{coghi2020pathwise}, where the authors show that in this setting, for globally Lipschitz drifts $B$, \eqref{eq:introddsde} is still the mean field limit of \eqref{eq:intro-particle-system}, although the connection to the PDE \eqref{eq:intro-PDE} breaks down; they present many additional results including the use of common noise, large deviations and a central limit theorem.
Let us stress that the additive structure of the noise $Y$ is such that the DDSDE \eqref{eq:introddsde} in integral form is pathwise meaningful, without any need of stochastic integration; in particular the results apply to naturally non-Markovian, non-martingale choices of $Y$, e.g. fractional Brownian motion.

A fundamental feature of \cite{coghi2020pathwise} is the revisitation of a technique first introduced by Tanaka \cite{tanaka1984Limits}; it allows, by a clever transformation of the underlying probability space, to transfer any stability estimate available for \eqref{eq:introddsde} into a bound for the corresponding empirical measure $L^N(X^{(N)})$ of the system \eqref{eq:intro-particle-system}. The argument is quite robust and allows one to deduce the mean field convergence property under very mild assumptions, even when the family $\{Y^i\}_i$ is not i.i.d. but only requiring
\[
L^N(Y^{(N)})= \frac{1}{N} \sum_{i=1}^N \delta_{Y^i} \rightharpoonup \cL(Y) \quad \PP\text{-a.s.}
\]
together with some finiteness of moments.
The i.i.d. assumption indeed is quite strong, as it implies exchangeability of the system (see \cite[Section I.2]{sznitman1991topics} for a detailed discussion) and sometimes one would prefer to avoid it, see \cite[Example II]{delattre2016note}. In this regard, like \cite{coghi2020pathwise}, our work can be seen as another attempt to advertise the power of Tanaka's idea, which has found applications also in the rough path setting in \cite{cass_lyons_14, bailleul2021propagation}.

As the global Lipschitz assumption considered in \cite{coghi2020pathwise} is quite restrictive, our main goal is conduct a deeper analysis of the DDSDE \eqref{eq:introddsde} providing sufficient conditions for existence, uniqueness and stability, which will also imply the mean field limit. As our setting includes in particular the typical case $Y=\xi + \sqrt{2\eps}\, W$, it also serves as a review of many standard techniques used in McKean--Vlasov SDEs, including Yamada--Watanabe principles, compactness arguments and monotonicity assumptions; we believe one of the merits of this work it to highlight their extreme robustness, as they apply for any given $Y$. On the downside, since our results all apply in the case $\eps=0$, our noise can be very degenerate and we can never go beyond the classical Vlasov setting.

In this regard, another motivation for the current work is to provide a family of ``baseline results'' to be compared to the regularizing features of suitable nondegenerate choices of $Y$; in our second article \cite{GalHarMay_21fBm} we will study in detail the well-posedness of \eqref{eq:introddsde} for highly irregular $B$ when the noise is sampled as a fractional Brownian motion.

However, we stress that this work is more than just a survey of known facts; our aim is also to provide state-of-the-art results. To give an example, for the rest of the introduction we will present a specific result from the paper concerning mean field convergence for McKean--Vlasov equations with potentially degenerate noise.
We deal with a drift $B$ of convolutional structure $B_t(x,\mu)=(b_t\ast \mu)(x)$ and a noise $Y=\xi+W$ given by independent variables $(\xi,W)$ such that $\cL(\xi)(\dd x)= \rho(x) \dd x$ for some density function $\rho$; we make no assumption on $W$ other than continuity.
The corresponding DDSDE becomes
\begin{equation}\label{eq:intro-ddsde-convol}
  X_t = \xi + \int_0^t \int_{\RR^d} b_s(X_s-y) \,\cL(X_s)(\dd y) + W_t.
\end{equation}
We refer to Section \ref{sec:notation} below for details of some notation used in the next statement.
\begin{thm}\label{thm:intro-main}
Assume $b\in L^1_T W^{1,q}_x$, $\div b\in L^1_T L^\infty_x$ and $\rho\in L^p_x$ for $(p,q)\in (1,\infty)$ satisfying
\[
q>d, \quad \frac{1}{p}+\frac{1}{q}<1;
\]
then there exists a unique solution to \eqref{eq:intro-ddsde-convol}, which moreover satisfies
\begin{equation}\label{eq:intro-bound}
  \sup_{t\in [0,T]} \| \cL(X_t)\|_{L^p_x} <\infty.
\end{equation}
Suppose additionally that $\cL(\xi+W)\in \cP_r(C_T)$ for all $r<\infty$ and that we are given a sequence $\{Y^i, i\in \NN\}$ of $C_T$-valued random variables such that, for any $r<\infty$, 
\[
\limsup_{N\to\infty} \| L^N(Y^{(N)})(\omega) \|_r <\infty, \quad L^N(Y^{N})(\omega) \rightharpoonup \cL(\xi +W)\quad \text{for $\PP$-a.e. } \omega;
\]
then any solution $\{X^{i,N}\}$ to \eqref{eq:intro-particle-system} associated to $\{Y^i\}_i$ satisfies
\[
L^N(X^{N})(\omega) \rightharpoonup \cL(X)\quad \text{for $\PP$-a.e. } \omega.
\]
\end{thm}
The above statement applies in particular for the choice $W=0$, i.e. when $Y^i$ is taken as i.i.d. copies of $\xi$ and the result becomes a mean field limit for the corresponding Vlasov equation. Although singular drifts have already been considered in such settings, see \cite{hauray2007n,jabin2015particles}, to the best of our knowledge no general statements for $b\in L^1_T W^{1,q}_x$ are available in the literature; in any case, our proof technique is novel and quite simple, as we illustrate below.

As typical of transport PDEs, the assumption $\div b\in L^1_T L^\infty_x$ is used to propagate the regularity $\cL(X_0)\in L^p_x$ and construct a solution $X$ to \eqref{eq:intro-ddsde-convol} satisfying \eqref{eq:intro-bound} by compactness arguments; the regularity $b\in L^1_T W^{1,q}_x$ for $q>d$, together with the one-sided Lipschitz estimates recently obtained in \cite{caravenna2021directional}, is then used to show that this solution $X$ is stable around \textit{any other} solution $\tilde{X}$, not necessarily satisfying \eqref{eq:intro-bound}. Finally, Tanaka's argument allows to translate this information into an estimate for the corresponding particle system and obtain the conclusion.

We conclude this introduction by outlining the structure of the paper. In Section \ref{sec:notation} we briefly recall some notation, conventions and well-known facts used throughout; Section \ref{sec:LipWellPosed} shortly revisits the theory of DDSDEs with Lipschitz drifts, in order to derive useful estimates for the sequel. Section \ref{sec:BenchmarkResults} constitutes the main body of the paper, establishing sufficient conditions for existence and uniqueness in several situations for general $B$. Section \ref{sec:convol} provides some more refined results in the more specific convolutional case $B=b\ast \mu$. Finally, in Section \ref{sec:MFL} we present Tanaka's idea and apply it to all the cases previously considered, deriving mean field limit results for each one. The appendices \ref{app:approximation} and \ref{app:maximal} contain analytic tools used throughout the paper.
\subsection{Notations, conventions and well-known facts}\label{sec:notation}

Throughout the article we will always work on a finite time interval $[0,T]$, that may be arbitrarily large; in particular we do not investigate the long-time behaviour of solutions nor convergence to equilibrium.
We will use the convention that whenever there exists a positive constant $C$ such that $a\leq C b$ we write $a\lesssim b$; if the constant $C$ depends on a significant parameter $p$, we write $a\lesssim_p b$.

For a Banach space $E$ we write $C_TE := C([0,T];E)$ for the space of continuous $E$-valued maps $f:[0,T]\rightarrow E$.
Similarly, for $p\geq 1$ we write $L^p_TE:= L^p(0,T;E)$ for Bochner--Lebesgue space of maps of the same form.
When $E=\RR^d$ we simply write $C_T,\,L^p_T$ etc.

Given a map $f:\RR^d\rightarrow \RR^m$ we write $Df$ for the $d\times m$ matrix of first derivatives. For $p\in [1,\infty]$, we denote by $L^p(\RR^d;\RR^m),\,W^{1,p}(\RR^d;\RR^m)$ (or simply $L^p_x,\,W^{1,p}_x$ when it does not cause confusion) the classical Lebesgue and Sobolev spaces.
We treat the space $W^{1,\infty}_x$ as synonymous with the Lipschitz continuous functions.
Whenever $p \in [1,\infty]$ and where it will not cause confusion, we write $p'$ to denote the dual exponent to $p$, that is $1/p+1/p'=1$, with the interpretation $p=1\iff p'=\infty$.
Similarly, we write $C^0_b$ for the space of bounded, continuous functions from $\RR^d$ to $\RR^m$; $f\in C^1_b$ if $f$ and $D f$ belong to $C^0_b$.

Given a separable Banach space $E$, we denote by $\cP(E)$ the set of probability measures over $E$; we write $\mu_n\rightharpoonup \mu$ for weak convergence of measures; i.e. convergence against continuous bounded functions.

Given $\mu,\nu\in \cP(E)$, $\Pi(\mu,\nu)$ stands for the set of all possible couplings of $(\mu,\nu)$, i.e. the subset of $\cP(E\times E)$ with first and second marginals given respectively by $\mu$ and $\nu$. For any $p\in [1,\infty)$, we define
\begin{equation*}
  d_p(\mu,\nu):=\inf_{m\in \Pi(\mu,\nu)} \bigg(\int_{E\times E} \|x-y\|^p_{E}\, m(\dd x,\dd y)\bigg)^{1/p}
\end{equation*}
which is a well defined quantity (possibly taking value $+\infty$). By \cite[Theorem 4.1]{villani2008optimal}, an optimal coupling $\bar{m}\in \Pi(\mu,\nu)$ realizing the above infimum always exists. Noting that we may equally describe the set $\Pi(\mu,\nu)$ as the collection of all $E$-valued random variables $X,\,Y$ defined on the same probability space, such that $\cL(X)=\mu$ and $\cL(Y)=\nu$, we obtain the frequently used inequality
\begin{equation*}
  d_p(\mu,\nu) \leq \EE[\|X-Y\|_E^p]^{1/p}\quad \forall\, (X,Y) \in \Pi(\mu,\nu).
\end{equation*}
We define $\cP_p(E)$ to be set of $p$-integrable probability measures; that is, $\mu\in \cP(E)$ s.t.
\begin{equation*}
  \| \mu\|_p := \bigg( \int_{E} \|x\|_E^p \, \mu(\dd x)\bigg)^{1/p}<\infty. 
\end{equation*}
It is well known that $d_p(\mu,\nu)<\infty$ for $\mu,\nu\in \cP_p(E)$ and that $(\cP_p(E),d_p)$ is a complete metric space, usually referred to as the $p$-Wasserstein space on $E$; let us stress however that our definition of $d_p(\mu,\nu)$ holds for all $\mu,\nu\in\cP(E)$. We recall that, given a sequence $\{\mu_n\}_n\subset \cP_p(E)$, $d_p(\mu_n,\mu)\to 0$ is equivalent to $\mu_n\rightharpoonup \mu$ weakly and $\| \mu_n\|_p\to \|\mu\|_p$, see \cite[Theorem 6.9]{villani2008optimal}.

Given $\mu\in \cP(\RR^d)$, with a slight abuse of notation we will write $\mu\in L^q(\RR^d)$ (or simply $L^q_x$) for $q\in [1,\infty]$ to indicate that $\mu$ admits a density $\mu(\dd x)=\rho(x)\dd x$ with respect to the $d$-dim. Lebesgue measure, such that $\rho\in L^q(\RR^d)$. 

Throughout the article, whenever not mentioned explicitly, we will consider an underlying probability space $(\Omega,\cF,\PP)$, where the $\sigma$-algebra $\cF$ is $\PP$-complete, and $\cL(X)$ is the push-forward measure of a random variable $X$ on this probability space, i.e. $\cL_{\PP}(X)=\cL(X)=\PP\circ X^{-1}=:X \#\PP$.
For $p\geq 1$, we will frequently consider $E$-valued random variables in the space $L^p(\Omega,\cF,\PP;E)$ which for notational simplicity we will denote by $L^p_\Omega E$.
We say that a sequence of random variables $\{X^n\}_{n}$ converges to $X$ in law if $\cL(X^n)\rightharpoonup \cL(X)$.
\section{Well-Posedness Under Lipschitz Assumptions}\label{sec:LipWellPosed}
We start by treating DDSDEs with Lipschitz continuous drifts. Although this case was already treated in \cite{coghi2020pathwise}, we revisit it in order to derive useful a priori estimates for the sequel; to avoid repetitiveness, we will also present alternative proofs to those of \cite{coghi2020pathwise}.

In this section we consider the DDSDE
\begin{equation}\label{eq:MKV}
X_t = \int_0^t B_s(X_s,\cL(X_s))\dd s + Y_t
\end{equation}
under the assumption that $B:[0,T]\times\RR^d\times\cP_p(\RR^d)\to\RR^d$ satisfies
\begin{equation}\label{eq:ass basic1}
|B_t(x,\mu)-B_t(x',\nu)|\leq g_t\, (|x-x'|+d_p(\mu,\nu)),\quad |B_t(x,\mu)|\leq h_t\, (1+|x|+\|\mu\|_p)
\end{equation}
for some $p\in [1,\infty)$ and some $g,h\in L^1(0,T;\RR^d)$.

Although the functions $g,h$ could be unified, we keep them separate to highlight the differing roles they play in the following estimates. Let us also point out that if $B$ is of the form $B_t(x,\mu)=\int b_t(x,y)\mu(\dd y)$, then sufficient conditions for \eqref{eq:ass basic1} to hold for $p=1$ (and thus for any $p\in [1,\infty)$) are
\begin{equation*}
|b_t(x,y)-b_t(x',y')| \leq g_t (|x-x'|+|y-y'|),\quad |b_t(x,y)|\leq h_t (1+|x| + |y|).
\end{equation*}

We refrain for now from providing a rigorous definition of solution to \eqref{eq:MKV}, which will be discussed in detail in Section \ref{sec:BenchmarkResults}; for now we note that thanks to the assumptions on $B$, the integral appearing in \eqref{eq:MKV} is pathwise meaningful and thus, given $Y\in L^p(\Omega,\cF,\PP;C_T)=: L^p_\Omega C_T$ for some probability space $(\Omega,\cF,\PP)$, we say that $X\in L^p_\Omega C_T$ solves \eqref{eq:MKV} if the identity holds $\PP$-a.s.

\begin{prop}\label{prop:MKV basic result}
Assume $B$ satisfies \eqref{eq:ass basic1}; then for any $Y\in L^p_\Omega C_T$ there exists a unique $X\in L^p_\Omega C_T$ solving \eqref{eq:MKV}. Moreover given $Y^1,Y^2$ as above, denoting by $X^1,X^2$ the associated solutions, it holds that
\begin{equation}\label{eq:MKV basic est2}
\EE[\| X^1 -X^2\|_{C_T}^p]^{1/p} \leq e^{2 \|g\|_{L^1_T}} \, \EE[\| Y^1-Y^2\|_{C_T}^p]^{1/p}.
\end{equation}
\end{prop}

\begin{proof}
Given $Y$ as above, we define a map $\cI$ on $L^p_\Omega C_T$, by
\[
\cI(X)_{\,\cdot\,} := \int_0^\cdot B_s(X_s,\cL(X_s))\dd s + Y_{\,\cdot\,};
\]
it is immediate to check, using assumption \eqref{eq:ass basic1}, that $\cI$ maps $L^p_\Omega C_T$ into itself. Let us endow $C_T$ with the equivalent norm
\[
\| x \tilde{\|}_{C_T} := \sup_{t\in [0,T]} \Big\{ e^{ -4 \int_0^t g_r \dd r}\, |x_s|\Big\},
\]
and by extension $L^p_\Omega C_T$ with the norm $\|\cdot \tilde{\|}_{L^p C_T}$. We claim that $\cI$ is a contraction on $(L^p_\Omega C_T, \| \cdot\tilde{\|}_{L^p C_T})$, which implies existence and uniqueness of a solution $X$ by the Banach--Caccioppoli theorem.

By the definition of $\| \cdot\tilde{\|}_{C_T}$, for any $X^1,X^2\in L^p_\Omega C_T$,
\[
d_p(\cL(X^1_t),\cL(X^2_t))
\leq \EE[ |X^1_t-X^2_t|^p]^{1/p} \leq e^{4\int_0^t g_r\,\dd r} \| X^1-X^2\tilde{\|}_{L^p C_T};
\]
therefore for any $X^1,X^2$, by assumption \eqref{eq:ass basic1} it holds that
\begin{align*}
e^{-4\int_0^t g_r \dd r} |\cI(X^1)_t - \cI(X^2)_t| & \leq e^{-4\int_0^t g_s \dd s} \int_0^t |B_s(X^1_s,\cL(X^1_s)) - B_s(X^2_s,\cL(X^2_s))| \dd s\\
& \leq e^{-4\int_0^t g_r \dd r} \int_0^t g_s\, [ |X^1_s-X^2_s| + d_p(\cL(X^1_x),\cL(X^2_s))] \dd s\\
& \leq e^{-4\int_0^t g_r \dd r} \bigg( \int_0^t g_s\, e^{4 \int_0^s g_r\dd r} \dd s\bigg)\, [\| X^1-X^2\tilde{\|}_{C_T} + \| X^1 -X^2\tilde{\|}_{L^p C_T}]\\
& \leq \frac{1}{4}\, [\| X^1-X^2\tilde{\|}_{C_T} + \| X^1 -X^2\tilde{\|}_{L^p C_T}].
\end{align*}
Taking first the supremum over $t$ on the l.h.s., then the $L^p_\Omega$-norm on both sides, we obtain
\[
\| \cI(X^1)-\cI(X^2)\tilde{\|}_{L^p C_T} \leq \frac{1}{2} \| X^1-X^2\tilde{\|}_{L^p C_T}
\]
showing that $\cI$ is a contraction.

Now let $X^i$ be solutions associated to $Y^i$ for $i=1,2$. Their difference satisfies
\begin{align*}
|X^1_t - X^2_t|
& \leq \int_0^t |B_s(X^1_s,\cL(X^1_s))-B_s(X^2_s,\cL(X^2_s))| \dd s + |Y^1_t-Y^2_t|\\
& \leq \int_0^t g_s \bigg (\sup_{r\leq s} |X^1_r-X^2_r| + \EE\Big[\sup_{r\leq s} |X^1_r-X^2_r|^p\Big]^{1/p} \bigg) \dd s + \| Y^1-Y^2\|_{C_T};
\end{align*}
taking the $L^p_\Omega$-norm on both sides, applying Minkowski's integral inequality, we arrive at
\begin{align*}
\EE\Big[\sup_{r\leq t} |X^1_r-X^2_r|^p\Big]^{1/p} \leq \int_0^t 2 g_s\, \EE\Big[\sup_{r\leq s} |X^1_r-X^2_r|^p\Big]^{1/p} \dd s  + \| Y^1-Y^2\|_{C_T}.
\end{align*}
Inequality \eqref{eq:MKV basic est2} then follows from an application of Gr\"onwall's lemma.
\end{proof}

So far we have regarded $X,Y$ as $C_T$-valued random variables on an abstract probability space $(\Omega,\cF,\PP)$; as the result holds on any such space, we can choose it to be $(C_T, \cB(C_T), \mu^Y)$, where $\mu^Y=\cL(Y)$ and $Y$ is realized as the canonical process.

Let us define the drift $\bar{b}(s, x):= B_s(x,\cL(X_s))$; by assumption \eqref{eq:ass basic1}, $\bar{b}$ also satisfies a standard Lipschitz-type condition, so by classical ODE theory, for any $\omega\in C_T$ there exists a unique solution $x\in C_T$ to
\[
x_\cdot = \int_0^\cdot \bar{b}(s,x_s)\dd s+ \omega_\cdot
\]
which can be expressed as $x_t = \Phi(t,\omega)$ for some continuous map $\Phi:[0,T]\times C_T\to C_T$.

Since $X$ is a solution to \eqref{eq:MKV}, it satisfies
\[
X_\cdot(\omega) = \int_0^\cdot \bar{b}(s,X_s(\omega))\dd s + Y_\cdot(\omega) \quad \text{ for } \PP\text{-a.e. } \omega\in\Omega
\]
implying that it satisfies $X(\omega) = \Phi(\cdot,Y(\omega))$. In particular, since $X$ can be expressed as a measurable function of $Y$, its law $\cL(X)=\mu^X = \Phi \#\mu^Y$ on $C_T$ is entirely determined by $\mu^Y$. As a consequence, we can regard the solution map $S^B$ coming from Proposition \ref{prop:MKV basic result} not only as a map $L^p_\Omega C_T$ to itself, but also as a map from $\cP_p(C_T)$ to itself, $S^B(\mu^Y)=\mu^X$, thus independent of the specific probability space in consideration.

We then obtain the following estimate, recovering the stability results from \cite{coghi2020pathwise}.

\begin{cor}
Assume $B$ satisfies \eqref{eq:ass basic1} and let $S$ be defined as above, then
\begin{equation*}
d_p(S(\mu^1),S(\mu^2)) \leq e^{2\| g\|_{L^1_T}} \, d_p(\mu^1,\mu^2)\quad \forall\, \mu^1,\mu^2\in \cP_p(C_T).
\end{equation*}
\end{cor}

\begin{proof}
Given any coupling $m\in \cP_p(C_T\times C_T)$ of $\mu^1,\mu^2$, considering $Y^1,Y^2$ as the canonical variables on $C_T\times C_T$, by Proposition \ref{prop:MKV basic result} we have
\[
d_p(S(\mu^1,S(\mu^2)) \leq \EE[\|X^1-X^2\|_{C_T}^p]^{1/p} \leq e^{2\| g\|_{L^1_T}} \EE[\| Y^1-Y^2\|_{C_T}^p]^{1/p};
\]
the conclusion follows by choosing $m$ to be the optimal coupling for $d_p(\mu^1,\mu^2)$.
\end{proof}

By the above discussion, in what follows we could always deal with measures $\mu\in \cP_p(C_T)$; however for the sake of computations we prefer to keep working with random variables $X,Y$ defined on an abstract probability space (which can be taken as the canonical one if needed).

Given $h\in L^1(0,T;\RR^d)$, we denote by $f_h$ the associated modulus of continuity, namely for any $\delta>0$,
\[f_h(\delta) := \sup_{s,t\in [0,T], |t-s|<\delta} \int_s^t |h_r|\dd r;
\]
let us also define, for any $x\in C_T$, the associate seminorm
\[
\llbracket x\rrbracket_h := \sup_{s,t\in [0,T], s\neq t} \frac{|x_t-x_s|}{f_h(|t-s|)}
\]
\begin{lem}\label{lem:MKV basic apriori estimates}
Let $B$ satisfy \eqref{eq:ass basic1}, $Y\in L^p_\Omega C_T$; then there exist a constant $C$, only depending on $\| h\|_{L^1_T}$, such that the solution $X$ to \eqref{eq:MKV} satisfies
\begin{equation}\label{eq:MKV basic apriori1}
\EE[ \| X\|_{C_T}^p]^{1/p} \leq C ( 1 + \EE[\| Y\|_{C_T}^p]^{1/p})
\end{equation}
as well as the pathwise bound
\begin{equation}\label{eq:MKV basic apriori2}
\| X(\omega)\|_{C_T} \leq C\, \big(1+\| Y(\omega)\|_{C_T}\big) \quad\text{for $\PP$-a.e. }\omega\in\Omega.
\end{equation}
Moreover, $X=Z+Y$ for some $Z\in L^p_\Omega C_T$ satisfying $Z_0=0$ and
\begin{equation}\label{eq:MKV basic apriori3}
\EE [\, \llbracket Z\rrbracket_h ^p]^{1/p} \leq C\, \EE[\, \| Y\|_{C_T}^p ]^{1/p}.
\end{equation}
\end{lem}
\begin{proof}
Let $X$ be a solution, then by assumption \eqref{eq:ass basic1}
\begin{equation}\label{eq:basic-proof}\begin{split}
\sup_{r\leq t} |X_r|
& \leq \int_0^t h_s (1+ |X_s| + \| \cL(X_s)\|_p )\dd s + |Y_t|\\
& \leq \int_0^t h_s \bigg(1+ \sup_{r\leq s} |X_r| + \EE\Big[\sup_{r\leq s} |X_r|^p\Big]^{1/p}\bigg) \dd s + \| Y\|_{C_T}.
\end{split}\end{equation}
Taking the $L^p_\Omega$-norm on both sides, applying Minkowski's integral inequality, we find
\begin{equation*}
\EE\Big[\sup_{r\leq t} |X_r|^p\Big]^{1/p} \leq \int_0^t h_t\bigg(1+2 \EE\Big[\sup_{r\leq s} |X_r|^p\Big]^{1/p}\bigg) \dd s + \EE[\| Y\|_{C_T}^p]^{1/p};
\end{equation*}
inequality \eqref{eq:MKV basic apriori1} then follows by Gr\"onwall's lemma, with $C= \exp(2\|h\|_{L^1_T}) (1+\|h\|_{L^1_T})$.

On the other hand, we can apply Gr\"onwall's lemma at a pathwise level to \eqref{eq:basic-proof} to find
\begin{align*}
\| X(\omega)\|_{C_T} \leq e^{\|h\|_{L^1_T}} \bigg(\| h\|_{L^1_T} + \EE[\| X\|_{C_T}^p]^{1/p} + \| Y(\omega)\|_{C_T}\bigg)
\end{align*}
which together with \eqref{eq:MKV basic apriori1} implies \eqref{eq:MKV basic apriori2}, for a suitable choice of $C$.

Finally, since $Z_t = \int_0^t B_s(X_s,\cL(X_s))$, applying again \eqref{eq:ass basic1} we find
\begin{align*}
|Z_t-Z_s| \leq \int_s^t |h_r| \dd r (1+ \| X\|_{C_T} + \EE[\| X\|_{C_T}^p]^{1/p});
\end{align*}
combined with \eqref{eq:MKV basic apriori1}, \eqref{eq:MKV basic apriori2} this implies the pathwise bound
\begin{align*}
\llbracket Z(\omega) \rrbracket \leq 1 + \| X(\omega)\|_{C_T} + \EE[\| X\|_{C_T}^p]^{1/p} \lesssim_h 1 + \| Y(\omega)\|_{C_T} + \EE[\| Y\|_{C_T}^p]^{1/p}.
\end{align*}
Inequality \eqref{eq:MKV basic apriori3} readily follows taking the $L^p_\Omega$-norm on both sides.
\end{proof}

\section{Refined criteria for existence and uniqueness}\label{sec:BenchmarkResults}

In many applications one cannot expect the drift $B$ to be globally Lipschitz, thus making the results of Section \ref{sec:LipWellPosed} not applicable.
Motivated by this fact, in this section we present several alternative conditions to establish existence and/or uniqueness of solutions to \eqref{eq:MKV}.

The first part of this section is focused on the different concepts of solution and the relations between them; Section \ref{subsec:ContDrift} provides a general result of existence of weak solutions under continuity and linear growth assumptions; Section \ref{subsec:uniqueness} deals with several refinements of the globally Lipschitz condition which are still sufficient to establish uniqueness. While we treat existence and uniqueness as separate problems, we provide example well-posedness results that combine the two. We leave open the possibility that ad-hoc existence or uniqueness results may be found in more specific cases, complementing the general results contained here.

The equations considered here are of the general form 
\begin{equation}\label{eq:GenMckeanSDE}
  X_t = \int_0^t B_s(X_s,\cL(X_s))\dd s + Y_t \quad \forall\, t\in [0,T]
\end{equation}
where $B : \RR_+\times \RR^d \times \cP_p(\RR^d)\rightarrow \RR^d$ is a measurable map, for $p\in [1,\infty)$, and $Y$ is a continuous process with finite $p$-moment. We detail a variety of assumptions on $B$ below.

\begin{defn}[Weak Solutions I]\label{def:WeakSolsI}
Given a drift $B$ as above and $\mu^Y \in \cP_p(C_T)$, we say that a tuple $(\Omega,\cF,\PP;X,Y)$, consisting of a probability space $(\Omega,\cF,\PP)$ and a measurable map $(X,Y):\Omega \rightarrow C_T\times C_T$, is a weak solution to the DDSDE \eqref{eq:GenMckeanSDE} on $[0,T]$ if:
\begin{enumerate}
  \item $\cL(Y) = \mu^Y$;
  \item $\sup_{t\in [0,T]} \| \cL(X_t)\|_p <\infty$;
  \item $\PP$-a.s. $\int_0^T |B_t(X_t,\cL(X_t))| \dd t <\infty$;
  \item the relation \eqref{eq:GenMckeanSDE} holds $\PP$-a.s., the integral being interpreted in the Lebesgue sense.
\end{enumerate}
We will refer to both $\mu^Y$ and $Y$ as the input data of the DDSDE.
\end{defn}

We readapt here the classical concepts of existence and uniqueness for SDEs to our setting.

\begin{defn}[Uniqueness in Law]\label{def:WeakUnique}
Given $\mu^Y\in \cP(C_T)$, we say that weak uniqueness holds for the DDSDE \eqref{eq:GenMckeanSDE} if for any pair of weak solutions $(\Omega^i,\cF^i,\PP^i;X^i,Y^i)$, $i=1,2$, such that $\cL_{\PP^i}(Y^i)=\mu^Y$, one has $\cL_{\PP^1}(X^1)=\cL_{\PP^2}(X^2)\in \cP(C_T)$.
\end{defn}

\begin{defn}[Pathwise Uniqueness]\label{def:pathWUni}
Given $\mu^Y\in \cP(C_T)$, we say that pathwise uniqueness holds for the DDSDE \eqref{eq:GenMckeanSDE} if for any two solutions $X^1,X^2$ defined on the same probability space $(\Omega,\cF,\PP)$ and w.r.t. to the same input process $Y$, it holds that $X^1=X^2$ $\PP$-a.s.
\end{defn}

\begin{defn}[Strong Existence]\label{def:StrongExist}
Given $\mu^Y\in\cP(C_T)$, we say that strong existence holds for the DDSDE associated to $\mu^Y$ if there exists a solution $(\Omega,\cF,\PP,X,Y)$ such that the process $X$ is adapted to the filtration $\cF^Y_t = \sigma\{Y_s: s \in [0,t]\}$.
\end{defn}

It follows from the definition that if strong existence holds, we can construct a solution $X$ on \textit{any} given probability space $(\Omega,\cF,\PP;Y)$ on which a process with $\cL_\PP(Y)=\mu^Y$ is defined.

We now provide several links between the various concepts of existence and uniqueness for DDSDEs; we start with a result in the style of the Yamada--Watanabe theorem \cite[Prop. 1]{yamada1971uniqueness}, showing that pathwise uniqueness implies uniqueness in law.

\begin{prop}\label{prop:baby-yamada-watanabe}
Let $p\in [1,\infty)$, $\mu^Y\in \cP_p(C_T)$, $B:[0,T]\times \RR^d\times\cP_p(\RR^d)\to \RR^d$ measurable; if pathwise uniqueness holds for the DDSDE, then uniqueness in law holds as well.
\end{prop}

\begin{proof}

Let $(\Omega^i,\cF^i,\PP^i;X^i,Y^i)$, $i=1,2$ be two weak solutions to \eqref{eq:GenMckeanSDE} with $\cL(Y^1)=\cL(Y^2)=\mu^Y$.
In order to prove uniqueness in law, it suffices to construct a coupling $(\tilde{X}^1,\tilde{X}^2)$ of $(X^1,X^2)$ such that $(\tilde{X}^1,\tilde{X}^2)$ solves \eqref{eq:GenMckeanSDE} on the same probability space and with same input $\tilde{Y}$; indeed once this is done, pathwise uniqueness implies $\cL_{\PP^1}(X^1)=\cL_{\tilde{\PP}}(\tilde{X}^1) =\cL_{\tilde{\PP}}(\tilde{X}^2) = \cL_{\PP^2}(X^2)$.

To construct the aforementioned coupling, we follow the proof of \cite[Prop. 1]{yamada1971uniqueness} based on disintegration of measures, for which we refer the reader to \cite[Thm. 5.3.1]{ambrosio_gigli_savare_08}.

Set $\mu^i := \cL(X^i,Y^i) \in \cP(C_T\times C_T)$ for $i=1,2$; since $C_T$ is Polish, we can disintegrate $\mu^i$ as $\{\mu^i_y\}_{y \in C_T}\subset \cP(C_T)$, $\mu^i_y$ representing the law of $X^i$ given $Y^i=y$. Define a measure $\bar{\mu} \in \cP(C_T\times C_T\times C_T)$ by setting
\begin{equation*}
  \dd \bar{\mu}(x^1,x^2,y) :=\dd\mu^1_{y}(x^1)\dd \mu^2_{y}(x^2)\dd\mu^Y(y)
\end{equation*}
and consider the probability space $(\bar{\Omega}, \bar{\cF}, \bar\mu)$ given by $\Omega=C_T\times C_T\times C_T$, $\bar{\cF}$ the $\bar{\mu}$-completion of $\cB(C_T\times C_T\times C_T)$. We also define the projection maps $\pi_{x^1,x^2}(x^1,x^2,y)=(x^1,x^2)$, and similarly $\pi_{x^i,y}$ for $i=1,2$.

We claim that $(\bar{\Omega}, \bar{\cF}, \bar{\mu}; \pi_{x^1,x^2})$ is the desired coupling. By construction,
\[
\cL_{\bar{\PP}}(\pi_{x^i,y}) = \dd \mu^i_y (x^i) \dd \mu^Y(y) = \cL_{\PP^i}(X^i,Y^i) \quad \text{for } i=1,2.
\]
Moreover the DDSDE relation \eqref{eq:GenMckeanSDE} can be written as $Y^i = F^i(X^i)$ for some measurable maps $F^i:C_T\to C_T$ given by $F^i(X^i)_\cdot := X^i - \int_0^\cdot B_s(X^i_s,\cL_{\PP^i}(X^i_s)) \dd s$; the same relations must then hold for $\bar\mu$-a.e. $(x^1,x^2,y)$, namely
\begin{align*}
  x^i_t
  = \int_0^t B_s(x^i_s, \cL_{\PP^i}(X^i_t))\dd s + y_t
  = \int_0^t B_s(x^i_s, \cL_{\bar{\mu}}(\pi_{x^i_t}))\dd s + y_t \quad \text{for } i=1,2
\end{align*}
where the property $\cL_{\bar{\mu}}(\pi_{x^i_t})= \cL_{\PP^i}(X^i_t)$ comes from the identity $\cL_{\bar{\mu}}(\pi_{x^i,y}) = \cL_{\PP^i}(X^i,Y^i)$. Overall this shows that $\pi_{x^1,x^2}$ is a coupling of $X^1$, $X^2$ solving the DDSDE \eqref{eq:GenMckeanSDE} with same input $\pi_y$, $\cL_{\bar{\mu}}(\pi_y)=\mu^Y$.
\end{proof}

\begin{rem}\label{rem:baby-yamada-watanabe}

The proof above also contains the nontrivial information that, given any two weak solutions $(\Omega^i,\cF^i,\PP^i;X^i,Y^i)$, we may construct a coupling of them. This is also true in the case $\mu^{Y^1}=\cL_{\PP^1}(Y^1)\neq\cL_{\PP^2}(Y^2)=\mu^{Y^2}$, assuming we already have a coupling $m\in \cP(C_T\times C_T)$ of $(\mu^{Y^1},\mu^{Y^2})$ available. Following the notation of the proof, in this case one can take
\[
\dd \bar{\mu}(x^1,x^2,y^1,y^2) := \dd \mu^1_{y^1}(x^1) \dd \mu^2_{y^2}(x^2) m(\dd y^1, \dd y^2)
\]
which defines $\bar{\mu}\in \cP(C_T\times C_T\times C_T\times C_T)$; $(\pi_{x^1,y^1},\pi_{x^2,y^2})$ is then a coupling of $(X^1,Y^1), (X^2,Y^2)$.
\end{rem}

The next proposition is in the style of \cite[Theorem 6.3]{wang2018distribution}.

\begin{prop}\label{prop:equivalence-weak-strong}
Let $p\in [1,\infty)$, $B:[0,T]\times \RR^d\times\cP_p(\RR^d)\to \RR^d$ measurable.
Suppose that for any $\mu^Y\in \cP_p(C_T)$ and any measurable $\mu:[0,T]\to\cP_p(\RR^d)$, strong existence and both pathwise uniqueness and uniqueness in law hold for the SDE
\begin{equation}\label{eq:reference-SDE}
  X_t = \int_0^t B_s(X_s,\mu_s)\dd s + Y_t.
\end{equation}
Then weak existence (resp. uniqueness in law) holds for the DDSDE \eqref{eq:GenMckeanSDE} if and only if strong existence (resp. pathwise uniqueness) holds for it. 
\end{prop}

\begin{proof}

Strong existence always implies weak existence and by Proposition \ref{prop:baby-yamada-watanabe}, pathwise uniqueness implies weak uniqueness.

Suppose now that there exists a weak solution $(\Omega,\cF,\PP;X,Y)$ to the DDSDE, then $X$ is also a solution to the SDE \eqref{eq:reference-SDE} for the choice $\mu_t=\cL(X_t)$.
By hypothesis, we can construct a strong solution $\bar{X}$ to the SDE on the canonical space $(C_T, \cB(C_T),\mu^Y)$;
since weak uniqueness holds for the SDE, $\cL_{\PP}(X)=\cL_{\mu^Y} (\bar{X})$, which implies that $\bar{X}$ is a strong solution to the DDSDE.

Finally, assume that weak uniqueness holds and let $X^1$, $X^2$ be two solutions defined on the same probability space and with same input $Y$.
Then $\cL(X^1_t)=\cL(X^2_t)=:\mu_t$ and they solve the same SDE with drift $(t,x)\mapsto B_t(x,\mu_t)$, so by pathwise uniqueness for the SDE it holds $X^1=X^2$ $\PP$-a.s.
\end{proof}
\subsection{Existence}\label{subsec:ContDrift}
Here we study existence of solutions to \eqref{eq:GenMckeanSDE} for continuous drifts $B$ satisfying suitable linear growth conditions.
\begin{ass}\label{ass:GenLinearGrowth}
Given $p\in [1,\infty)$, $B:[0,T]\times \RR^d \times \cP_p(\RR^d)\rightarrow \RR^d$ is measurable and s.t.:
\begin{itemize}
  \item[i.] For any $t\in [0,T]$, $B_t:\RR^d \times \cP_p(\RR^d)\rightarrow \RR^d$ is uniformly continuous on bounded sets.
  \item[ii.] There exists $h\in L^1_T$ such that for all $(t,x,\mu) \in [0,T]\times \RR^d\times \cP_p(\RR^d)$ one has
\begin{equation}\label{eq:TimeInhomLinearGrowth}
 |B_t(x,\mu)| \leq h_t\left(1+|x|+ \|\mu\|_{p} \right).
\end{equation}
\end{itemize}
%
\end{ass}
\begin{rem}\label{rem:equiv w def}
Let $\delta_0$ denote the Dirac delta in $0$, which belongs to $\cP_p(\RR^d)$ for any $p\in [1,\infty)$; it's then easy to check that $\| \mu\|_p = d_p(\mu,\delta_0)$ for any $\mu\in \cP_p$. Therefore if $B$ satisfies Assumption \ref{ass:GenLinearGrowth}, it also satisfies Assumption \ref{ass:TimeInHomLG} from Appendix \ref{app:approximation} in the metric space $E=\RR^d\times \cP_p(\RR^d)$.
\end{rem}

The next result can be seen as an extension of Peano's theorem to DDSDEs.

\begin{prop}\label{prop:GenExistenceUnderGrowth}
Let $B$ satisfy Assumption \ref{ass:GenLinearGrowth} for some $p\in[1,\infty)$ and $\mu^Y\in \cP_p(C_T)$; then there exists a weak solution to \eqref{eq:GenMckeanSDE} in the sense of Definition \ref{def:WeakSolsI}. 
Moreover, any weak solution $(\tilde{\Omega},\tilde{\cF},\tilde{\PP}; \tilde{X},\tilde{Y})$ satisfies the pathwise bound
\begin{equation}\label{eq:apriori-bound-exist}
  \| \tilde{X}\|_{C_T} \leq C (1+\| \tilde{Y}\|_{C_T}) \quad \tilde{\PP}\text{-a.s.}
\end{equation}
for some constant $C=C(\| h\|_{L^1_T})$.
\end{prop}
\begin{proof}
Estimate \eqref{eq:apriori-bound-exist} follows in exactly the same manner as \eqref{eq:MKV basic apriori2}, since it only relies on the linear growth of $B$ and Gr\"onwall's lemma; thus we focus on establishing existence.

Thanks to Assumption \ref{ass:GenLinearGrowth} and Remark \ref{rem:equiv w def}, we can apply Proposition \ref{prop:TimeInhomLipApprox} from Appendix \ref{app:approximation} to $B$; we then find a sequence of drifts $\{B^n\}_{n}$ such that each $B^n:[0,T]\times \RR^d\times \cP_p(\RR^d)\rightarrow \RR^d$ satisfies condition \eqref{eq:ass basic1} with the same growth function $h_t$ and $B^n$ is Lipschitz with constant $g^n>0$ independent of $t$.
Denote by $Y$ the canonical process associated to $\mu^Y$; since $\mu^Y \in \cP_p(C_T)$, we may apply Proposition \ref{prop:MKV basic result} to obtain the existence of a sequence $\{X^n\}_{n}$, each $X^n$ solving \eqref{eq:GenMckeanSDE} with $B$ replaced by $B^n$.

Lemma \ref{lem:MKV basic apriori estimates} furnishes us with suitable a priori estimates for the family $\{X^n\}_{n}$: there exists $C=C(\|h\|_{L^1_T})$ such that
\begin{equation}\label{eq:PeanoAPriori}
  \EE\big[\|X^n\|_{C_T}^p \big]^{1/p} + \EE\big[\,\llbracket Z^n\rrbracket_h^p\big]^{1/p} \leq C\Big (1+ \EE[ \| Y\|_{C_T}^p]^{1/p} \Big)
\end{equation}
where the estimate is uniform over $n$ and $Z^n$ are given by the relation $X^n = Z^n + Y$; we recall that $\llbracket \,\cdot\,\rrbracket_h$ was the seminorm associated to the modulus of continuity $f_h$ defined in Section \ref{sec:LipWellPosed}.

By Arzel\`a--Ascoli, subsets of $C_T$ of the form $\{\omega\in C_T : \omega_0=0, \llbracket \omega\rrbracket_h \leq R\}$ are compact; therefore by the moment estimate \eqref{eq:PeanoAPriori}, the family of measures $\{\cL(Z^n)\}_n$ are tight in $C_T$.
Since $\cL(Y)$ is also tight on $C_T$ (being a probability measure on a Polish space), we deduce that the family $\{\cL(Y,X^n)\}_n$ of probability measures on $C_T\times C_T$ is tight as well.

We can then apply Prokhorov's theorem to extract a (not relabelled) subsequence such that $\{\cL(X^n,Y)\}_n$ converge weakly to a limit measure $\mu$. By Skorokhod's theorem we can construct a new probability space $(\tilde{\Omega},\tilde{\cF},\tilde{\PP})$, carrying a family of random variables $(\tilde{X}^n,\tilde{Y}^n)$, $(\tilde{X},\tilde{Y})$, such that $(\tilde{X}^n(\tilde \omega),\tilde{Y}^n(\tilde \omega))\to (\tilde{X}(\tilde \omega),\tilde{Y}(\tilde \omega))$ for $\tilde{\PP}$-a.e. $\tilde\omega$, $\mathcal{L}_\PP(X^n,Y)=\mathcal{L}_{\tilde{\PP}}(\tilde{X}^n,\tilde{Y}^n)$ and $\mu = \cL_{\tilde{\PP}} (\tilde{X},\tilde{Y})$. It remains to show that $(\tilde{\Omega}, \tilde\cF,\tilde\PP; \tilde{X},\tilde{Y})$ is a weak solution to \eqref{eq:GenMckeanSDE}.

First of all, observe that $\cL_{\tilde{\PP}}(\tilde{Y}^n) = \cL_\PP(Y)=\mu^Y$ and $\cL_{\tilde{\PP}}(\tilde{Y}^n)\rightharpoonup \cL_{\tilde{\PP}}(\tilde{Y})$ so that $\cL_{\tilde{\PP}}(\tilde{Y})=\mu^Y$, as well as $\EE_{\tilde{\PP}} [\| \tilde{Y}^n\|_{C_T}^p] \to \EE_{\tilde{\PP}} [\| \tilde{Y}\|_{C_T}^p]$; together with the convergence $\tilde{Y}^n \to \tilde{Y}$ $\tilde{\PP}$-a.s., we find that $\tilde{Y}^n \to \tilde{Y}$ in $L^p_{\tilde{\Omega}} C_T$.
Since $(\tilde{Y}^n,\tilde{X}^n)$ have the same law as $(Y,X^n)$, they are still solutions to the DDSDEs associated to $B^n$; in particular they satisfy the pathwise bound \eqref{eq:MKV basic apriori2}. This information together with the convergence $\tilde{Y}^n \to \tilde{Y}$ in $L^p_{\tilde{\Omega}} C_T$ implies that $\{\tilde{X}^n\}_n$ is uniformly $p$-integrable. However, since we also have $\tilde{X}^n\to \tilde{X}$ $\tilde{\PP}$-a.s., by Vitali's convergence theorem $\tilde{X}^n\to \tilde{X}$ in $L^p_{\tilde{\Omega}} C_T$. In particular $d_p(\cL_{\tilde{\PP}}(\tilde{X}^n_t), \cL_{\tilde{\PP}}(\tilde{X}_t))\to 0$ and overall we have the $\tilde{\PP}$-a.s. convergence $(\tilde{X}_t^n,\cL(\tilde{X}_t^n))\rightarrow(\tilde{X}_t,\cL(\tilde{X}_t))$ in $\RR^d\times \cP_p(\RR^d)$ for any $t\in [0,T]$.

The above convergence also implies that $\tilde{X}$ satisfies the pathwise bound \eqref{eq:apriori-bound-exist}; together with assumption \eqref{eq:TimeInhomLinearGrowth} this shows that points $(2)$ and $(3)$ of Definition \ref{def:WeakSolsI} are satisfied.
To conclude that identity \eqref{eq:GenMckeanSDE} holds $\tilde{\PP}$-a.s. it then suffices to show that $\int_0^\cdot B^n_s(\tilde{X}^n_s, \cL(\tilde{X}^n_s)) \dd s$ converge to $\int_0^\cdot B_s(\tilde{X}_s, \cL(\tilde{X}_s)) \dd s$.

Since $B^n$ converge to $B$ uniformly on compact sets and $(\tilde{X}_t^n,\cL(\tilde{X}_t^n))\rightarrow(\tilde{X}_t,\cL(\tilde{X}_t))$ $\tilde{\PP}$-a.s., it holds $B^n(\tilde{X}_t^n,\cL(\tilde{X}_t^n))\to B(\tilde{X}_t,\cL(\tilde{X}_t))$ $\tilde{\PP}$-a.s. as well.
Moreover since $B^n$ and $B$ both satisfy the growth assumption \eqref{eq:TimeInhomLinearGrowth}, we have
\begin{align*}
  \sup_{t\in [0,T]} \bigg|\int_0^t B^n_s(\tilde{X}^n_s, \cL(\tilde{X}^n_s)) \dd s- \int_0^t B_s(\tilde{X}_s, \cL(\tilde{X}_s) \dd s \bigg|
  \leq \int_0^T |B^n_s(\tilde{X}^n_s, \cL(\tilde{X}^n_s)) - B_s(\tilde{X}_s, \cL(\tilde{X}_s))| \dd s
\end{align*}
with the second integrand satisfying the uniform bound
\begin{align*}
  |B^n_s(\tilde{X}^n_s, \cL(\tilde{X}^n_s)) - B_s(\tilde{X}_s, \cL(\tilde{X}_s)|
  & \leq \| h\|_{L^1_T} (2+ \| \tilde{X}^n\|_{C_T} + \| \cL(\tilde{X}^n)\|_p + \|\tilde{X}\|_{C_T} + \| \cL(\tilde{X})\|_p)\\
  & \lesssim_h 1 + \| \tilde{Y}^n\|_{C_T} + \| \tilde{Y}\|_{C_T} + \| \cL(\tilde{Y})\|_p
\end{align*}
where we used again the a priori bound \eqref{eq:apriori-bound-exist}, as well as $\cL(\tilde{Y}^n)=\cL(\tilde{Y})$. But $\|\tilde{Y}^n\|_{C_T}\to \|\tilde{Y}\|_{C_T}$ $\tilde{\PP}$-a.s., thus the previous integrands must be uniformly integrable.

Since $[0,T]$ is finite, we may now apply again Vitali's convergence theorem to give that
\begin{align*}
  \int_0^\cdot B^n(s,\tilde{X}^n_s,\cL(\tilde{X}^n_s))\,\dd s \to \int_0^\cdot B(s,\tilde{X}_s,\cL(\tilde{X}_s))\,\dd s\quad \tilde{\PP}\text{-a.s.}
\end{align*}
which implies the conclusion.
\end{proof}

\begin{rem}
Condition \textit{i.} from Assumption \ref{ass:GenLinearGrowth} can be verified in a variety of situations;
it suffices to check that for any $t\in [0,T]$, the map $(x,\mu)\mapsto B_t(x,\mu)$ is continuous under a weak topology of $\RR^d\times\cP_p(\RR^d)$ for which bounded balls are compact (as we can then apply Heine--Cantor theorem).

For instance, if $B_t(\cdot,\cdot)$ is continuous on $\RR^d\times \cP_p(\RR^d)$, then it's uniformly continuous on bounded balls on $\RR^d\times \cP_q(\RR^d)$ for any $q>p$, due to the compact embedding $\cP_q\hookrightarrow \cP_p$. Similarly, in the case $p=1$ it's enough to require
\begin{equation*}
  |B_t(x,\mu)-B_t(y,\nu)|\lesssim_t F\Big(|x-y| + \inf_{m\in\Pi(\mu,\nu)} \int_{\RR^{2d}} |x'-y'|^\theta\, m(\dd x',\dd y')\Big)
\end{equation*}
for some $\theta\in (0,1)$ and some continuous $F$ such that $F(0)=0$. On the use of more abstract weak topologies guaranteeing sequential compactness of bounded balls of $\cP(\RR^d)$ we also refer to \cite[Section B]{Gartner1988}.
\end{rem}

\begin{rem}\label{rem:growth-vs-motononicity}
The reader might wonder if we can replace the growth condition in \eqref{eq:TimeInhomLinearGrowth} by a monotonicity assumption like $\langle B_t(x,\mu),x \rangle \lesssim 1+ |x|^2$; there are several difficulties in doing so.
On a technical side, one would need to control $\int_0^t \langle X_s, \dot Y_s\rangle$, which require either $Y$ to be of bounded variation or some more refined integration theory (It\^o, rough paths) to be available; recall that here we are not imposing any assumption on $Y$ other than continuity.

More importantly, already in the case of standard ODEs there are examples of finite time blow-up for suitable choices of $b$ and continuous $Y$, see \cite[Section 3.3]{cox2013local}; the monotonicity assumption can be replaced by more refined criteria, see \cite{riedel2017rough} and the recent work \cite{bonnefoi2020priori}, which however don't seem to transfer easily to the distribution-dependent setting.
\end{rem}

\subsection{Uniqueness}\label{subsec:uniqueness}

In this section we provide various conditions for uniqueness of solutions to the DDSDE \eqref{eq:GenMckeanSDE}, thus extending already known results for ODEs and SDEs to the the distribution dependent case. We do not aim at maximal generality, as the criteria from Sections \ref{subsec:osgood}-\ref{subsec:loc-lipschitz} could be combined together to create new ones; rather we aim to illustrate how many classical assumptions on $B$ still work perfectly without any need for $Y$ to be Brownian or Markovian.

Whenever possible, we try to keep uniqueness statements separate from their existence counterparts (which in our setting always comes from Proposition \ref{prop:GenExistenceUnderGrowth}) and only afterwards unite them together in a well-posedness result.
Depending on the DDSDE in consideration, ad-hoc existence results which go beyond Assumption
\ref{ass:GenLinearGrowth} might be available; they can still be combined with the uniqueness criteria presented here, which are thus of independent interest.

\subsubsection{Osgood-type condition}\label{subsec:osgood}
It is well known that classical ODEs admit a unique solution under Osgood type conditions on the drift (originally proposed in \cite{Osgood1898}); we include here an analogous statement for DDSDEs. 

\begin{ass}\label{ass:Osgood}
It holds that
\begin{equation}\label{eq:OsgoodDrift}
  |B_t(x,\mu)-B_t(y,\nu)|\leq h_t\,f(|x-y|+d_1(\mu,\nu)),
\end{equation}
for all $(x,y)\in \RR^{2d}$, $(\mu,\nu)\in \cP_1(\RR^d)^2$ and $t\in [0,T]$, where $h\in L^1_T$ and $f:\RR_+ \rightarrow \RR_+$ is a modulus of continuity satisfying 
\begin{equation}\label{eq:OsgoodIntegralCondition}
  \int_0^\eps \frac{1}{f(u)}du=\infty\quad {\rm for\,\, any} \,\,\varepsilon>0.
\end{equation}
\end{ass}
\begin{rem}
By a result of Stechkin which can be found as \cite[Lem. A]{medvedev2001concave}, given any modulus of continuity $f$, there exists another concave modulus $\tilde{f}$ satisfying the two-sided bound $\tilde{f}\leq f\leq 2\tilde{f}$. Thus we can always assume the modulus $f$ appearing in \eqref{eq:OsgoodIntegralCondition} to be concave, which also implies it is increasing and subadditive.
\end{rem}

Before proving our uniqueness results, we 
recall Bihari's inequality, see e.g. \cite{Beesack1976}.

\begin{prop}\label{prop:Bihari}
Let $f:[a,b]\to (0,+\infty)$ be a continuous, monotone, strictly positive function and let $x$ and $h$ be two functions on an interval $[t_0,t_0+T]$ such that $x([t_0,t_0+T])\subset [a,b]$, $h\geq 0$.
Assume that for some $\alpha\in [a,b]$ one has
\begin{equation*}
  x_t\leq \alpha+\int_{t_0}^t h_s\, f(x_s)\dd s,\quad t\in [t_0,t_0+T];
\end{equation*}
then
\begin{equation}\label{eq:bihari}
  x_t\leq G^{-1}\left(G(a)+\int_{t_0}^t h_s\dd s\right),\quad t\in [t_0,t_0+\tilde{T}],
\end{equation}
where
\begin{equation}\label{eq:bihari-G}
  G(u):=\int_{x_0}^u \frac{1}{f(r)}\dd r
\end{equation}
and $\tilde{T}\leq T$ is chosen such that
\begin{equation*}
  G(a)+\int_{t_0}^t h(s)\dd s\in {\rm Dom}(G^{-1}),\quad \forall t\in [t_0,t_0+\tilde{T}]. 
\end{equation*}
\end{prop}

\begin{rem}\label{rem:domain}
It's easy to check, using the definition of $G$, that functions of the form $a\mapsto G^{-1}(G(a)+\kappa)$ as appears on the r.h.s. of \eqref{eq:bihari} do not depend on the choice of $x_0$ (as long as $a$ belongs to the corresponding domain). Since any concave modulus of continuity $f$ can grow at most linearly at infinity, under the assumption \eqref{eq:OsgoodIntegralCondition} we deduce that $G$ satisfies $\lim_{x\to 0} G(x) = -\infty$, $\lim_{x\to +\infty} G(x)=+\infty$ and ${\rm Dom}(G^{-1})=\RR$, so that we can take $\tilde{T}=T$.
\end{rem}

We now provide a stability estimate for solutions defined on the same probability space; given two drifts $B^1,B^2:[0,T]\times\RR^d\times \cP_1(\RR^d)\to \RR^d$, we introduce the notation
\[
\|B^1-B^2\|_\infty:= \sup_{(s,x,\mu)\in [0,T]\times \RR^d\times \cP_1(\RR^d) } |B^1_s(x,\mu)-B^2_s(x,\mu)|.\]
\begin{prop}
\label{prop:strong stability}
Let $B^1$ and $B^2$ be two drifts satisfying Assumption \ref{ass:Osgood} for the same $f$, $h$.
Assume $X^1$ and $X^2$ are two solutions to \eqref{eq:GenMckeanSDE} defined on the same probability space with input $(Y^1,B^1)$ and $(Y^2,B^2)$ respectively. Then 
\begin{equation}\label{eq:strong osgood stability}
\EE[\|X^1-X^2\|_{C_T}] \leq G^{-1}\left(G\left( \EE[\|Y^1-Y^2\|_{C_T}] + \|B^1-B^2\|_\infty T\right)+2\|h\|_{L^1_T}\right)
\end{equation}
where $G$ is defined in \eqref{eq:bihari-G}. In particular, pathwise uniqueness and uniqueness in law hold for the DDSDE.
\end{prop}

\begin{proof}
We begin by observing that
\begin{equation*}
  X^1_t-X^2_t = \int_0^t B^1_s(X^1_s,\cL(X^1_s))-B^2_s(X^2_s,\cL(X^2_s))\dd s +Y^1_t-Y^2_t;
\end{equation*}
adding and subtracting $B^1_s(X^2_s,\cL(X^2_s))$, applying Assumption \ref{ass:Osgood} and the monotonicity of $f$, we see that
\begin{align*}
  \sup_{r \leq t} |X^1_r - X^2_r|
  & \leq \| Y^1-Y^2\|_{C_T} + T \| B^1-B^2\|_{\infty} + \int_0^t h_s \, f(|X^1_s-X^2_s| + d_1(\cL(X^1_s),\cL(X^2_s))\, \dd s\\
  & \leq \| Y^1-Y^2\|_{C_T} + T \| B^1-B^2\|_{\infty} + \int_0^t h_s\, f\bigg(\sup_{r\leq s} |X^1_r-X^2_r| + \EE\Big[\,\sup_{r\leq s} |X^1_r-X^2_r|\Big]\bigg)\, \dd s
\end{align*}
 Taking expectation on both sides, using subadditivity and concavity of $f$ along with Jensen's inequality, we get that
\begin{align*}
  \EE\Big[\sup_{r \leq t} |X^1_r - X^2_r| \Big]
  & \leq \EE[\,\| Y^1-Y^2\|_{C_T} ] + T \| B^1-B^2\|_{\infty} + \int_0^t 2 h_s\, f\bigg( \EE\Big[ \sup_{r \leq s} |X^1_r - X^2_r|\Big]\bigg)\,\dd s.
\end{align*}
Applying Proposition \ref{prop:Bihari} readily gives estimate \eqref{eq:strong osgood stability}.

Regarding the last statement, observe that for $Y^1=Y^2$, $B^1=B^2$ estimate \eqref{eq:strong osgood stability} implies $X^1\equiv X^2$, proving pathwise uniqueness. Uniqueness in law then follows from Proposition \ref{prop:baby-yamada-watanabe}.
\end{proof}


Combining Assumptions \ref{ass:GenLinearGrowth} and \ref{ass:Osgood} we can obtain strong existence and uniqueness of solutions; moreover we can recast the stability estimate \eqref{eq:strong osgood stability} into a corresponding bound on the Wasserstein distance between the laws of solutions.
\begin{thm}
\label{thm:osgood-stability}
Let $B^1$ and $B^2$ satisfy Assumption \ref{ass:GenLinearGrowth}, with $p=1$ and Assumption \ref{ass:Osgood} for the same $f$, $h$ and let $\mu^{Y^1},\,\mu^{Y^2}\in \cP_1(C_T)$ be two input laws. Then for $i=1,\,2$ there exist unique weak solutions, $\mu^{X^i}$ associated to $(B^i, \mu^{Y^i})$. Furthermore, there exists a continuous, increasing function $M:\RR_+\to\RR_+$, $M(0)=0$, only depending on $f$ and $h$, such that
\begin{equation}\label{eq:osgood stability}
d_1(\mu^{X^1},\mu^{X^2})\leq M\left( d_1(\mu^{Y^1},\mu^{Y^2}) + T\|B^1-B^2\|_\infty \right). 
\end{equation}
\end{thm}
\begin{proof}
Since both $B^1$ and $B^2$ satisfy Assumption \ref{ass:GenLinearGrowth}, weak existence holds by Proposition \ref{prop:GenExistenceUnderGrowth}; Assumption \ref{ass:Osgood} and Proposition \ref{prop:equivalence-weak-strong} then imply strong existence as well. Pathwise uniqueness and uniqueness in law instead are consequences of Proposition \ref{prop:strong stability}.

Now let $(Y^1,Y^2)$ be an optimal coupling for $d_1(\mu^{Y^1},\mu^{Y^2})$, which induces a coupling $(X^1,X^2)$ for $\mu^{X^1}, \mu^{X^2}$ by strong existence and uniqueness in law. Define $M:(0,+\infty)\to (0,+\infty)$ by
\begin{equation*}
  M(r) := G^{-1}\big(G(r) + 2 \| h\|_{L^1_T}\big)
\end{equation*}
for $G$ defined as in \eqref{eq:bihari-G}; by the properties $\lim_{x\to 0} G(x)=-\infty$, $\lim_{x\to-\infty}G^{-1}(x) = 0$ we deduce that $M$ extends uniquely and continuously at $r=0$ with $M(0)=0$.

But then inequality \eqref{eq:strong osgood stability} and the choice of optimal coupling $(Y^1,Y^2)$ readily implies
\begin{align*}
  d_1(\mu^{X^1},\mu^{X^2}) \leq \EE[\|X^1-X^2\|_{C_T}] \leq M\big( \EE[\|Y^1-Y^2\|_{C_T}] + T \| B^1-B^2\|_\infty\big)
\end{align*}
and the conclusion follows.
\end{proof}

Let us provide some relevant examples of drifts $B$ satisfying Assumption \ref{ass:Osgood} above. 
\begin{ex}\label{ex:TrueMKVOsgood}
Consider the case of a true McKean--Vlasov drift, i.e. $B$ of the form
\begin{equation*}
  B_t(x,\mu)=\int_{\RR^d} b_t(x,y)\,\mu(\dd y)
\end{equation*}
for some measurable function $b:[0,T]\times\RR^d\times\RR^d\to\RR^d$. Then Assumption \ref{ass:Osgood} is satisfied if there exist $h\in L^1_T$, and a (concave) modulus of continuity $f$ satisfying \eqref{eq:OsgoodIntegralCondition} such that
\begin{equation*}
  |b_t(x,z)-b_t(y,z')|\leq h_t \,f(|x-y|+|z-z'|).
\end{equation*}

Indeed, for any $\mu,\nu\in\cP_1(\RR^d)$ and any coupling $m\in\Pi(\mu,\nu)$ it holds
\begin{align*}
  |B_t(x,\mu)-B_t(y,\nu)| 
  & = \Big|\int_{\RR^d} b_t(x,z) \mu(\dd z) - \int_{\RR^d} b_t(y,z') \nu(\dd z')\Big|\\
  & \leq \int_{\RR^{2d}} |b_t(x,z)-b_t(y,z')|\, m(\dd z,\dd z')\\
  & \leq h_t \int_{\RR^{2d}} f(|x-y|+|z-z'|)\,m(\dd z,\dd z')\\
  & \leq h_t\,\Big[f(|x-y|) + \int_{\RR^{2d}} f(|z-z'|)\, m(\dd z,\dd z')\Big]\\
  & \leq h_t\,\Big[ f(|x-y|) + f \Big( \int_{\RR^{2d}} |z-z'|\, m(\dd z,\dd z')\Big)\Big]
\end{align*}
where we used monotonicity, subadditivity and concavity of $f$ as well as Jensen inequality; minimizing over $m\in\Pi(\mu,\nu)$ gives \eqref{eq:OsgoodDrift}.

Of particular relevance are convolutional drifts $B_t(x,\mu)=(b_t\ast\mu)(x)=\int b_t(x-z)\,\mu(\dd z)$, in which case the above condition reduces to
\begin{align*}
  |b_t(x)-b_t(y)|\leq h_t\, f(|x-y|),\quad \forall \,(t,x,y)\in [0,T]\times \RR^d\times\RR^d.
\end{align*}
\end{ex}
\subsubsection{Monotone drifts}\label{subsec:monotone}

Another classical assumption in ODE theory, that goes beyond the Lipschitz setting, is a monotonicity condition (sometimes also referred to as one-sided Lipschitz condition).
Similar assumptions in the DDSDE setting have been employed for example in \cite[Sec. 5]{Gartner1988}; see also the more recent works \cite{wang2018distribution} (Assumption (H2) therein) and \cite{dos2019freidlin}. Finally, let us mention that monotonicity may be interpreted as a Lyapunov condition on the energy $|\cdot|^2$ and is therefore also close in spirit to the assumptions considered in \cite{mehri2019weak}, \cite{hammersley2018mckean}.
 
We consider the following working assumption for the drift $B$. 

\begin{ass}\label{ass:monotone-coeff}
Given $p\in (1,\infty)$, there exists an $h\in L^1_T$ such that
\begin{equation}\label{eq:monotone-coeff}
  \langle x-y, B_t(x,\mu)-B_t(y,\nu)\rangle \leq h_t \big[ |x-y|^2 + |x-y|\,d_p(\mu,\nu)\big]
\end{equation}
uniformly over $x,y\in\RR^d$, and $\mu,\nu\in \cP_p(\RR^d)$.
\end{ass}

\begin{prop}
Let $B:[0,T]\times \RR^d \times \cP_p(\RR^d)\rightarrow \RR^d$ be a measurable map satisfying Assumption \ref{ass:monotone-coeff} for some $p\in (1,\infty)$; then pathwise uniqueness and uniqueness in law hold for \eqref{eq:GenMckeanSDE}.
\end{prop}

\begin{proof}
By virtue of Proposition \ref{prop:baby-yamada-watanabe}, we only need to establish pathwise uniqueness.
Let $X^1,X^2$ be two solutions to \eqref{eq:GenMckeanSDE} defined on the same probability space and with same input $Y$; setting $\mu^i_t =\cL(X^i_t)$ for $i=1,2$, it holds that
\[
X^i_t = \int_0^t B_s(X^i_s,\mu^i_s)\,\dd s + Y_t
\]
and the difference $X^1-X^2$ is an absolutely continuous path. By Assumption \ref{ass:monotone-coeff} we have that
\begin{align*}
  \frac{\dd}{\dd t} |X^1_t-X^2_t|^p
  & = p\,|X_t^1-X_t^2|^{p-2}\, \langle X^1_t-X^2_t, B_t(X^1_t,\mu^1_t)- B_t(X^2_t,\mu^2_t)\rangle\\
  & \leq p\, h_t \big[ |X^1_t-X^2_t|^p+ |X^1_t-X^2_t|^{p-1} \, d_p(\mu^1_t,\mu^2_t)\big]\\
  & \lesssim_p h_t\left( |X^1_t-X^2_t|^p + d_p(\mu^1_t,\mu^2_t)^p\right)
\end{align*}
where in the last passage we used the basic inequality $ab\lesssim a^{\frac{p}{p-1}} + b^p$. Since $X^1_0=Y_0=X^2_0$, integrating and taking expectations we obtain
\begin{align*}
  \EE\big[|X^1_t-X^2_t|^p\big]
  \lesssim \int_0^t h_s\Big( \EE[|X^1_s-X^2_s|^p] + d_p(\mu^1_s,\mu^2_s)^p\Big) \dd s
  \lesssim \int_0^t h_s\EE[|X^1_s-X^2_s|^p] \dd s
\end{align*}
and by Gr\"onwall we deduce that $\EE[|X^1_t-X^2_t|^p] =0$ for all $t\in [0,T]$; since $X^i$ are continuous paths,
pathwise uniqueness follows.
\end{proof}

\begin{rem}
If $p\geq 2$, we can further weaken condition \eqref{eq:monotone-coeff} by requiring instead
\begin{equation}\label{eq:monotone-coeff-2}
   \langle x-y, B_t(x,\mu)-B_t(y,\nu)\rangle \leq h_t \big[ |x-y|^2 + d_p(\mu,\nu)^2\big]
\end{equation}
uniformly over $t,x,y,\mu,\nu$. Indeed, going through the same computations as in the proof above, gives
\begin{align*}
  \frac{\dd}{\dd t} |X^1_t-X^2_t|^p
  & \leq p\, h_t \big[ |X^1_t-X^2_t|^p+ |X^1_t-X^2_t|^{p-2} \, d_p(\mu^1_t,\mu^2_t)^2\big]\\
  & \lesssim_p h_t \big[ |X^1_t-X^2_t|^p + d_p(\mu^1_t,\mu^2_t)^p \big]
\end{align*}
where we used the basic inequality $ab\lesssim a^{\frac{p}{p-2}} + b^{\frac{p}{2}}$ which holds since $p/2\geq 1$.
\end{rem}

\begin{ex}\label{ex:monotone-1}
It is clear that any Lipschitz $B$ satisfies Assumption \ref{ass:monotone-coeff}; more generally, one can take $B_t(x,\mu)=F_t(x) + G(x,\mu)$ for globally Lipschitz $G$ and $F$ satisfying
\begin{equation}\label{eq:monotone-ex}
  \langle F_t(x)-F_t(y),x-y\rangle \leq h_t |x-y|^2.
\end{equation}
A similar computation holds for $B_t(x,\mu)=F_t(x)G(\mu)$ with $F$ as above, once we additionally impose $\| F_t\|_{L^\infty} \leq h_t$ and $G:\cP_p(\RR^d)\to\cP_p(\RR^d)$ is globally Lipschitz and bounded.

In the case $h\equiv C$ for some $C\in \RR$, condition \eqref{eq:monotone-ex} is satisfied if $F=-\nabla V$ for some $V:\RR^d\to\RR$ such that $x\mapsto V(x)+\lambda |x|^2/2$ is convex for some $\lambda\geq 0$; interesting cases include $F(x)=-x e^{|x|^2}$ and $F(x)= -|x|^{\gamma-1} x$ for $\gamma\geq 0$ where for $\gamma\in [0,1)$, $F$ is not even locally Lipschitz.
\end{ex}
Although Assumption \ref{ass:monotone-coeff} covers a wide range of drifts, sometimes it's natural to further weaken it to an \textit{integrated} monotonicity assumption.

\begin{lem}
Suppose $B:[0,T]\times\RR^d\times \cP_2(\RR^d)\to \RR^d$ satisfies
\begin{equation}\label{eq:monotone-integr-cond}
  \int_{\RR^{2d}} \langle x-y, B_t(x,\mu)-B_t(y,\nu)\rangle\, m(\dd x,\dd y) \leq h_t\,\int_{\RR^{2d}} |x-y|^2\, m(\dd x,\dd y)
\end{equation}
uniformly over all possible $t\in [0,T]$, $\mu,\nu\in \cP_2(\RR^d)$ and $m\in \Pi(\mu,\nu)$, for some integrable $h:[0,T]\to\RR$;
then pathwise uniqueness and uniqueness in law hold for the DDSDE.
\end{lem}

\begin{proof}
By virtue of Proposition \ref{prop:baby-yamada-watanabe}, we only need to establish pathwise uniqueness.
Let $X^1$,$X^2$ be two solutions defined on the same probability space $(\Omega,\cF,\PP)$ and with same input $Y$. It follows that $X^1-X^2$ is a process of bounded variation satisfying $\PP$-a.s.
\[
|X^1_t - X^2_t|^2
= 2\int_0^t\langle X^1_t- X^2_t, B_s(X^1_s,\cL(X^1_s))-B_s(X^2_s,\cL(X^2_s))\rangle \,\dd s \quad \forall\, t\in [0,T].
\]
Taking expectation and applying assumption \eqref{eq:monotone-integr-cond} for $m=\PP\circ (X^1_s,X^2_s)^{-1}\in \Pi(\cL(X^1_s),\cL(X^2_s))$, we arrive at
\begin{align*}
  \EE[|X^1_t - X^2_t|^2]
  \leq 2\int_0^t h_s\, \EE[|X^1_s-X^2_s|^2]\, \dd s
\end{align*}
and the conclusion follows by Gr\"onwall's lemma.
\end{proof}
\begin{ex}
Let $b:\RR^d\times \RR^d\to \RR^d$ be a function satisfying
\[
b(x,x')= - b(x',x), \quad
\langle b(x,x')-b(y,y'),x-x'-y+y'\rangle \leq C (|x-y|^2 + |x'-y'|^2)
\]
for some constant $C\in \RR$, uniformly over $x,x',y,y'\in \RR^d$;
then the McKean-Vlasov drift $B_t(x,\mu)=\int_{\RR^d} b(x,x')\,\mu(\dd x')$ satisfies condition \eqref{eq:monotone-integr-cond} with $h_t=C$.

Indeed, for any $m\in \Pi(\mu,\nu)$ it holds
\begin{align*}
  \int_{\RR^{2d}} \langle x-y, & B_t(x,\mu)-B_t(y,\nu)\rangle\, m(\dd x,\dd y)\\
  & = \int_{\RR^{2d}\times \RR^{2d}} \langle x-y, b(x,x')-b(y,y')\rangle \, m(\dd x,\dd y) m(\dd x',\dd y')\\
  & = \frac{1}{2} \int_{\RR^{2d}\times \RR^{2d}} \langle x-y-x'+y', b(x,x')-b(y,y')\rangle \, m(\dd x,\dd y) m(\dd x',\dd y')
\end{align*}
where the last passage follows by exchanging $(x,y)$ and $(x',y')$ and applying $b(x',x)=-b(x,x')$; by the monotonicity condition on $b$ we readily obtain the conclusion.

For convolutional drifts $B=b\ast \mu$ it suffices to require
\[
b(-x)=-b(x),\quad \langle x-y,b(x)-b(y)\rangle\leq 2C |x-y|^2 \quad \forall\, x,y\in\RR^d;
\]
as in Example \ref{ex:monotone-1}, this is true for instance if $b(x)=-\nabla V(x)$ for some even function $V:\RR^d\to\RR$ such that $x\mapsto V(x)+\lambda|x|^2$ is convex fr some $\lambda\geq 0$. This class includes drifts not satisfying the linear growth assumption \ref{ass:GenLinearGrowth}, but for which uniqueness still holds.

Let us mention that the special cases $b(x)=-\lambda |x|^{\gamma-1} x$ for $\lambda,\gamma>0$ are also the drift terms considered in \cite{wang2018distribution,fournier17landau,carrillo20landau}; the setting of these works however, also includes a non-linear diffusion term, associated to kernels of the form $a(x) = |x|^{\gamma-1} (|x|^2\,I_d - x \otimes x)$.
\end{ex}
It follows from the above results that if $B$ satisfies Assumption \ref{ass:GenLinearGrowth} and one between Assumption \ref{ass:monotone-coeff} or conditions \eqref{eq:monotone-coeff-2}, \eqref{eq:monotone-integr-cond}, then for any $\mu^Y\in \cP_p(C_T)$ there exists a unique solution $\mu^X\in \cP_p(C_T)$; this correspondence defines a solution map $S^B:\cP_p(C_T)\to \cP_p(C_T)$. Unfortunately we cannot establish Lipschitz continuity for $S^B$, the issue being similar to that of Remark \ref{rem:growth-vs-motononicity} and related to controlling $\int_0^t \langle X^1_s-X^2_s, \dd (Y^1_s-Y^2_s) \rangle$; we can still at least show continuity.

\begin{lem}\label{lem:monotone-sol-map}
Let the drift $B$ satisfy the above assumptions, then the solution map $S^B:\cP_p(C_T)\to \cP_p(C_T)$ is continuous.
\end{lem}

\begin{proof}
Given a sequence $\{\mu^{Y^n}\}_n$ with $\mu^{Y^n}\to \mu^Y$ in $\cP_p(C_T)$, we need to show that $\mu^{X^n}=S^B(\mu^{Y^n})\to S^B(\mu^Y)=\mu^X$. By the assumption and estimate \eqref{eq:apriori-bound-exist} it follows that $(\mu^{Y^n},\mu^{X^n})$ is tight in $\cP(C_T\times C_T)$, so by Prokhorov's theorem we can extract a (not relabelled) subsequence s.t. $(\mu^{Y^n},\mu^{X^n})\rightharpoonup (\mu^Y,\mu^Z)$ for a suitable $\mu^Z\in \cP(C_T)$.
By Skorokhod's theorem, we can work on a common probability space where all processes $(X^n,Y^n)$ and $(Z,Y)$ are well defined; using the fact that $Y^n\to Y \in L^p(\Omega;C_T)$, $X^n$ satisfy \eqref{eq:apriori-bound-exist} and arguing as in the proof of Proposition \ref{prop:GenExistenceUnderGrowth} we deduce that $X^n\to Z$ in $L^p(\Omega;C_T)$ and that $Z$ solves the DDSDE associated to $Y$.
But then by uniqueness $Z=X$ and $\mu^{X^n}\to \mu^X$ in $\cP_p(C_T)$; as the reasoning works for any possible subsequence, conclusion follows.
\end{proof}

\subsubsection{Locally Lipschitz drifts with growth conditions}\label{subsec:loc-lipschitz}

In the classical ODE and SDE setting, it is well known that uniqueness holds as soon as the drift is merely locally Lipschitz; the same is not true for DDSDEs, with explicit counterexamples given in \cite{scheutzow1987uniqueness}.
Since global Lipschitz assumptions are often too restrictive to work with, we provide here some uniqueness results for drifts with local Lipschitz constant not growing too fast; the price one has to pay is a strong integrability requirement on the input $\mu^Y$.
Similar type of conditions have been proposed in the classical work \cite{bolley2011stochastic} and more recently in \cite{wang2018distribution}, \cite{erny2021wellposedness}.

\begin{ass}\label{ass:local-lipschitz}
Given $\alpha> 0$ and $p\geq 1$; the drift $B:[0,T]\times\RR^d\times \cP(\RR^d)\to \RR^d$ and input data $\mu^Y$ satisfy the following:
\begin{enumerate}[label=\roman*.]
\item $|B_t(x,\mu)|\leq h_t (1+|x|+\| \mu\|_p)$, \label{it:LocalLipLinGrowth}
\item $|B_t(x,\mu)-B_t(y,\nu)| \leq g_t (|x-y|+d_p(\mu,\nu)) (1+|x|^\alpha+|y|^\alpha+\| \mu\|_p^\alpha+\| \nu\|_p^\alpha)$, \label{it:LocalLip}
\item $\int_{C_T} \exp(\lambda \|\omega \|_{C_T}^\alpha) \mu^Y(\dd \omega)<\infty$ for all $\lambda\in\RR$, \label{it:LocalLipExpMoments}
\end{enumerate}
for integrable functions $h,g\in L^1_T$.
\end{ass}

To obtain a stability estimate under Assumption \ref{ass:local-lipschitz},
we need the following simple lemma.
\begin{lem}\label{lem:ConvexExponential}
For any $\gamma \in (0,1)$ there exists a $C>0$ such that the map $\RR^d \ni x\mapsto e^{(C+|x|)^\gamma}$ is convex.
\end{lem}
\begin{proof}
We reduce to the one dimensional case by arguing component-wise. Then note that it suffices to chose $C=C(\gamma)>0$ such that $\gamma C^\gamma+\gamma-1\geq 0$ and to check that the resulting second derivative is non-negative.
\end{proof}
\begin{prop}\label{prop:local-lipschitz-stability}
Under Assumption \ref{ass:local-lipschitz}, strong existence, pathwise uniqueness and uniqueness in law hold for the DDSDE \eqref{eq:GenMckeanSDE} associated to $\mu^Y$. Moreover for any $q>p$ there exists a constant $C=C(\| h\|_{L^1_T}, \|g\|_{L^1_T}, \alpha, p, q)$ such that, for any two solutions $\mu^{X^1}$ and $\mu^{X^2}$ driven by the inputs $\mu^{Y^1}$ and $\mu^{Y^2}$ respectively, we have the stability estimate: 
\begin{equation}\label{eq:local lipschitz stability}
  d_p(\mu^{X^1},\mu^{X^2}) \leq C \bigg( \int_{C_T} \exp(C \|\omega \|_{C_T}^\alpha) \, (\mu^{Y^1}+\mu^{Y^2})(\dd \omega) \bigg)\,d_q(\mu^{Y^1},\mu^{Y^2}).
\end{equation}
\end{prop}

\begin{proof}
Assumption \ref{ass:local-lipschitz} implies that the hypothesis of Proposition \ref{prop:GenExistenceUnderGrowth} are met, thus weak existence holds. Moreover any solution $X$ satisfies the a priori estimate \eqref{eq:apriori-bound-exist},
namely $\| X\|_{C_T} \lesssim_h \| Y\|_{C_T}$.
%

Strong existence follows from weak existence and Proposition \ref{prop:equivalence-weak-strong}; by Proposition \ref{prop:baby-yamada-watanabe}, weak uniqueness and pathwise uniqueness are equivalent and will follow from the stability estimate.

Now let $(Y^1,Y^2)$ be an optimal coupling for $d_q(\mu^{Y^1},\mu^{Y^2})$; by Remark \ref{rem:baby-yamada-watanabe}, given any pair of weak solutions $\mu^{X^1},\mu^{X^2}$, we can construct an associated coupling $(X^1,X^2)$ and thus work with all variables defined on the same probability space. Let us set $\mu^i_t := \cL(X^i_t)$.

By point \ref{it:LocalLip} of Assumption \ref{ass:local-lipschitz}, the difference $X^1-X^2$ satisfies
\begin{align*}
|X^1_t-X^2_t|
& \leq |Y^1_t-Y^2_t|+ \int_0^t g_s\, (|X^1_s-X^2_s|+d_p(\mu^1_s,\mu^2_s))(1+|X^1_s|^\alpha+|X^2_s|^\alpha+ \| \mu^1_s\|_p^\alpha +\| \mu^2_s\|_p^\alpha)\, \dd s.
\end{align*}
Setting $Z_t := 1+|X^1_t|^\alpha+|X^2_t|^\alpha+ \| \mu^1_t\|_p^\alpha +\| \mu^2_t\|_p^\alpha$, it follows from the pathwise bound \eqref{eq:apriori-bound-exist} that
\begin{equation}\label{eq:proof-a-priori}
  \| Z\|_{C_T} \lesssim_{h,\alpha} 1+ \| Y^1\|_{C_T}^\alpha + \| Y^2\|_{C_T}^\alpha + \| \mu^{Y^1}\|_p^\alpha + \| \mu^{Y^2}\|_p^\alpha\quad \PP\text{-a.s.}; 
\end{equation}
moreover by Gr\"onwall, for any $t\in [0,T]$ we have the $\PP$-a.s. pathwise estimate
\begin{equation}\label{eq:bound w Z}
\begin{aligned}
\sup_{r\leq t}|X^1_r-X^2_r|
& \leq \exp\bigg(\int_0^t g_s Z_s \dd s\bigg)\left(\|Y^1-Y^2\|_{C_T}+ \int_0^t g_s\, d_p(\mu^1_s,\mu^2_s)\, Z_s \dd s\right)\\
& \lesssim \exp\Big((1+\| g\|_{L^1_T}) \| Z\|_{C_T}\Big) \left(\|Y^1-Y^2\|_{C_T}+\int_0^t g_s\, d_p(\mu^1_s,\mu^2_s)\,\dd s\right)
\end{aligned}
\end{equation}
%
Now taking the $L^p_\Omega$-norm on both sides, using an H\"older inequality on the r.h.s. (using the fact that $q>p$), we can find another constant $C'=C'(\| g\|_{L^1_T},p,q)$ such that
\begin{align*}
  \EE\Big[\,\sup_{r\leq t} |X^1_r - X^2_r|^p\Big]^{1/p} 
  & \lesssim \EE[\exp(C' \| Z\|_{C_T})]\, \left(\EE[\| Y^1-Y^2\|_{C_T}^q]^{1/q} +\int_0^t g_s\, d_p(\mu^1_s,\mu^2_s)\,\dd s\right)\\
  & \lesssim \EE[\exp(C' \| Z\|_{C_T})]\, \left(d_q(\mu^{Y^1},\mu^{Y^2}) +\int_0^t g_s\, \EE\Big[\,\sup_{r\leq s} |X^1_r - X^2_r|^p\Big]^{1/p}\,\dd s\right).
\end{align*}
Applying Gr\"onwall's lemma again, we deduce that
\begin{align*}
  d_p(\mu^{X^1},\mu^{X^2}) \leq \EE[\| X^1-X^2\|_{C_T}^p]^{1/p} \lesssim e^{\| g\|_{L^1_T}}\,\EE[\exp(C' \| Z\|_{C_T})]\, d_q(\mu^{Y^1},\mu^{Y^2}).
\end{align*}

Now consider the quantities $\exp(\|\mu^{Y^i}\|^\alpha_{p}) = \exp(\EE[\|Y^i\|^{p}_{C_T}]^{\alpha/p})$ for $i=1,\,2$. In the case $\alpha >p$ one can directly apply Jensen to give that $\exp(\EE[\|Y^i\|^{p}_{C_T}]^{\alpha/p}) \leq \EE [\exp(\|Y^i\|^\alpha_{C_T})]$. For $\alpha<p$, we first apply Lemma \ref{lem:ConvexExponential} with $\gamma = \alpha/p$, followed by Jensen to give, for some $C_\gamma>0$,
\begin{align*}
  \exp\left(\EE[\|Y^i\|^p_{C_T}]^{\alpha/p}\right) \leq \exp\left((C_\gamma+\EE[\|Y^i\|^{p}_{C_T}])^{\alpha/p}\right) &\leq \EE\left[\exp((C_\gamma+\|Y^i\|^p_{C_T})^{\alpha/p})\right]\\
  &\lesssim_{\alpha,p} \EE\left[\exp(\|Y^i\|^\alpha_{C_T})\right].
\end{align*}
In either case we have $\exp(\|\mu^{Y^i}\|^\alpha_{p}) \lesssim_{\alpha,p} \EE[\exp(\|Y^i\|^\alpha_{C_T})]$ and so combined with \eqref{eq:proof-a-priori} we finally obtain the stability estimate \eqref{eq:local lipschitz stability} for some new constant $C$.

As the estimate holds for any possible pair of solutions $X^1,X^2$, taking $Y^1=Y^2$ we deduce pathwise uniqueness and this completes the proof.
\end{proof}

\begin{rem}
It's clear from the proof above that Assumption \ref{ass:local-lipschitz} admits other variants, for instance we could instead require, \ref{it:LocalLipLinGrowth}, \ref{it:LocalLipExpMoments} and
\begin{enumerate}[label=\roman*$^\prime$.] \setcounter{enumi}{1}
\item $|B_t(x,\mu)-B_t(y,\nu)| \leq g_t (|x-y|+d_p(\mu,\nu)) (1+|x|^\alpha+|y|^\alpha+\langle e^{c\,|\cdot|^\alpha}, \mu\rangle +\langle e^{c\,|\cdot|^\alpha}, \nu \rangle)$
\end{enumerate}
for some $c>0$. The only difference is that in this case existence of solutions doesn't follow from a straightforward application of Proposition \ref{prop:GenExistenceUnderGrowth}. However, it's easy to check that, due to the available a priori estimates, the same compactness argument can be readapted in this setting.
\end{rem}

The next example provides concrete choices of $B$ satisfying Assumption \ref{ass:local-lipschitz}.

\begin{ex}\label{ex:growth}
If $Y=\xi+W$, where $\xi$ has Gaussian tails and $W$ is a continuous Gaussian process (e.g. fractional Brownian motion), then condition \ref{it:LocalLipExpMoments} is satisfied for any $\alpha\in [0,2)$.

Suppose $b:[0,T]\times \RR^d\times\RR^d\to \RR^d$ satisfies for some $\alpha>0$
\begin{align*}
  |b_t(x,y)|& \leq g_t(1+|x|+|y|),
  \\
  |b_t(x,x')-b_t(y,y')|&\leq g_t (|x-y|+|x'-y'|)(1+|x|^\alpha + |y|^\alpha +|x'|^\alpha +|y'|^\alpha),
\end{align*}
for some integrable function $g$. Then a few elementary computations reveal that 
\[
B_t(x,\mu):= \int_{\RR^d} b_t(x,x') \mu(\dd x')
\]
satisfies conditions \ref{it:LocalLipLinGrowth},\,\ref{it:LocalLip} for $p=\alpha+1$. Indeed by H\"older's and Minkowski's inequalities
\begin{align*}
  |B_t(x,\mu)-B_t(y,\nu)|
  &\leq \int_{\RR^{2d}}|b_t(x,x')-b_t(y,y')|m(\dd x',\dd y') 
  \\
  &\leq g_t \bigg( |x-y| +\Big(\int_{\RR^{2d}} |x'-y'|^p \, m(\dd x',\dd y')\Big)^{1/p}\bigg) \times\\
  & \quad \times \bigg( 1+ |x|^\alpha +|y|^\alpha + \Big(\int_{\RR^{2d}} (|x'|^\alpha +|y'|^\alpha)^{p'} \, m(\dd x',\dd y')\Big)^{1/p'} \bigg)\\
  &\leq g_t(|x-y|+d_p(\mu,\nu))(1+|x|^\alpha+|y|^\alpha+\|\mu\|_{p'\alpha}^\alpha+\|\nu\|_{p'\alpha}^\alpha)
\end{align*}
once we take $m$ to be an optimal coupling for $d_p(\mu,\nu)$; choosing $p=\alpha+1$, so that $p'\alpha=p$, we get the desired estimate for \ref{it:LocalLip} and condition \ref{it:LocalLipLinGrowth} can be checked similarly. 

The above assumption on $b$ is satisfied for instance if $D b$ satisfies a growth condition of the form $|D b_t (x,x')|\lesssim g_t (1+|x|^\alpha+|x'|^\alpha)$, for all $(x,x')\in \RR^{2d}$.
\end{ex}
\section{DDSDEs with convolutional structure}\label{sec:convol}
In this section we will focus on a particular subcase of \eqref{eq:GenMckeanSDE}, given by drifts with the convolutional structure $B_t(x,\mu)=(b_t\ast \mu)(x)$. Contrary to the previous sections, we now impose the noise term $Y$ to be of the form $Y_t=\xi + W_t$ where $\xi$ and $W$ are independent random variables taking respectively values in $\RR^d$ and $C_T$; the reason for this, which will become clearer later on, is to exploit the transport structure to show that some integrability features of $\cL(\xi)$ are propagated at positive times. Still, no assumption on $W$ are imposed apart from its continuity.

The DDSDE in consideration therefore becomes
\begin{equation}\label{eq:ConvMckeanSDE}
X_t
= \xi + \int_0^t (b_s\ast\mathcal{L}(X_s))(X_s)\,\mathd s+W_t
= \xi + \int_0^t \int_{\RR^d} b_s(X_s - y) \,\mathcal{L}(X_s)(\dd y) + W_t
\end{equation}
for some measurable $b:[0,T]\times\RR^d\to\RR^d$.

The convolutional structure arises naturally in the study of particle systems with pairwise interactions; it also allows to discuss kernels $b$ with poor spatial regularity, possibly unbounded and merely integrable.

Let us start by providing a suitable notion of solution to equation~\eqref{eq:ConvMckeanSDE}; as usual, we work for simplicity on a finite time interval $[0,T]$.

\begin{defn}[Weak Solutions II]\label{def:WeakSolsII}
Let $b:\RR_+\times \RR^d \rightarrow \RR^d$ be a measurable map. We say that a tuple $(\Omega,\cF,\PP;X,W,\xi)$, consisting of a probability space $(\Omega,\cF,\PP)$ and a measurable map $(X,W,\xi):\Omega \rightarrow C_T\times C_T\times \RR^d$ is a weak solution to \eqref{eq:ConvMckeanSDE} on $[0,T]$ if:
\begin{enumerate}
  \item For all $t>0$, $\cL(X_t)=\PP \circ X^{-1}_t=:\mu_t$ is absolutely continuous  w.r.t. Lebesgue.
  \item There exists a measurable representative of $(t,x)\mapsto |b_t|\ast \cL(X_t)(x)$ such that
  \begin{equation*}
    \int_0^T |b_s|\ast \cL(X_s)(X_s) \dd s <\infty,\quad \PP\text{-a.s..}
  \end{equation*}
  \item The relation \eqref{eq:ConvMckeanSDE} holds $\PP$-a.s., the integral being interpreted in the Lebesgue sense.
\end{enumerate}
\end{defn}
\begin{rem}\label{rem:MeasruableRep}
Definition \ref{def:WeakSolsII} is insensitive to the choice of measurable representative for $b_t\ast\cL(X_s)$ appearing in $(2)$. Consider any two measurable maps $f, \tilde{f} :[0,T]\times \RR^d\rightarrow \RR^d$ such that $f(t,x)=\tilde{f}(t,x)$ for Lebesgue a.e. $(t,x)$, then by condition $(1)$ and Fubini we have
\begin{align*}
  \EE\Big[ \int_0^T |f(s,X_s)-\tilde{f}(s,X_s)|\mathd t\Big] =\int_0^T \int_{\RR^d} |f(s,x)-\tilde{f}(s,x)| \,\cL(X_s) (\dd x) \dd s =0.
\end{align*}
As a consequence, if $\PP$-a.s. $\int_0^T |f(s,X_s)|\dd s<\infty$, then the same holds for $\tilde{f}$ and one has that $\int_0^\cdot f(s,X_s)\dd s= \int_0^\cdot \tilde f(s,X_s)\dd s$ as $C_T$-valued random variables. A similar argument shows that the definition is insensitive to modifications of the kernel $b$ on Lebesgue negligible sets; for this reason we can consider kernels belonging to equivalence classes like $L^1_T L^p_x$.
\end{rem}
In order for the requirement $(2)$ in Definition \ref{def:WeakSolsII} to be met, it is desirable to have some information on the integrability properties of $\cL(X_t)$. To derive it, we need to recall known facts from classical ODE and transport PDE theory.

Given $x_0\in \RR^d$, a drift $\bar{b}:[0,T]\times\RR^d\rightarrow \RR^d$ with continuous and bounded first derivative and a continuous path $\omega\in C_T$, there exists a unique solution $x\in C_T$ to the Cauchy problem
\begin{equation}\label{eq:RandomODE}
x_t = x_0 + \int_0^t \bar{b}(s,x_s)\mathd s + \omega_t \quad\forall\, t\in [0,T];
\end{equation}
similarly to Section \ref{sec:LipWellPosed}, we can define the solution map $x_t=\Phi(t,x_0;\omega)$ associated to the drift $\bar{b}$, which is a continuous map $\Phi:[0,T]\times \RR^d\times C_T\to \RR^d$. For fixed $(t,\omega)$, the map $x_0\mapsto \Phi(t,x_0;\omega)$ is a diffeomorphism of $\RR^d$, with Jacobian given by the formula
\begin{align*}
\det D \Phi(t,x_0;\omega)
= \exp \left(\int_0^t \div \bar{b}(s, \Phi(s,x_0;\omega))\mathd s \right).
\end{align*}
In particular we have the following two-sided estimate independent of $\omega$:
\begin{equation}\label{eq:PhiIncompressBound}
  0<\exp\big(-\| \div \bar{b}\|_{L^1_T L^\infty_x}\big) \leq \det D \Phi(t,x_0;\omega) \leq \exp\big(\| \div \bar{b}\|_{L^1_T L^\infty_x}\big)<\infty;
\end{equation}
similarly for $\det D\Psi(t,x_0;\omega)$, where $\Psi(t,\,\cdot\,;\omega)$ denotes the inverse of $x_0\mapsto \Phi(t,x_0;\omega)$.

On the other hand, we can consider the solution map as a continuous (thus measurable) map $(x_0,\omega)\mapsto \Phi(\,\cdot\,,x_0;\omega)$ from $\RR^d\times C_T$ into $C_T$; in particular if we now consider a random variable $(\xi,W)$ defined on a probability space $(\Omega,\cF,\PP)$, we consider the associated variable $\Phi(\,\cdot\,,\xi;W)$, which amounts to the solution of the random ODE problem \eqref{eq:RandomODE}.

With this considerations in mind, we can prove the following preliminary a priori estimate.

\begin{prop}\label{prop:LawLpApriori}
Let $b\in L^1_T C^1_b$, $p\in [1,\infty]$ and $(\xi,W)$ be independent input data s.t. $\cL(\xi)(\dd x)=\rho(x)\dd x$ with $\rho\in L^p_x$ for $p\in [1,\infty]$ and let $X$ denote the unique solution to \eqref{eq:ConvMckeanSDE}.
Then there exists a constant $C=C(p,\Vert \div b\Vert_{L^1_T L^\infty_x})$ such that
\begin{equation}\label{eq:LipKernelApriori}
\sup_{t\in [0,T]} \Vert \mathcal{L}(X_t)\Vert_{L^p_x} \leq C \| \rho\|_{L^p_x}
\end{equation}\\
Furthermore, only using the assumption that $\rho$ is a probability density, for any $\eps>0$ there exists a $\delta>0$, only depending on $\rho$ and $\| \div b\|_{L^1_T L^\infty_x} $, s.t.
\begin{equation}\label{eq:apriori-equiint}
  \sup_{t\in [0,T]} \cL(X_t)(A) \leq \eps \quad \forall\, A\subset \RR^d,\, |A|\leq \delta.
\end{equation}
\end{prop}

\begin{proof}
Let $X$ be the aforementioned solution, which exists by the results of Section \ref{sec:LipWellPosed}. Then $X$ is also a solution to the random ODE \eqref{eq:RandomODE} for the choice $\bar{b}_t(x) = b_t\ast \cL(X_t)(x)$ and we can represent it as $X_t = \Phi(t,\xi;W)$; without loss of generality we can take $(\xi,W)$ to be the canonical variables on the probability space $(\RR^d\times C_T, \cB(\RR^d\times C_T),(\rho \dd x)\otimes \mu)$ where $\mu=\cL(W)$. Moreover it holds that
\[
\| \div \bar{b}\|_{L^1_T L^\infty_x}
= \| (\div b_\cdot) \ast \cL(X_\cdot)\|_{L^1_T L^\infty_x}
\leq \| \div b \|_{L^1_T L^\infty_x}.
\]
For any measurable and bounded $f:\RR^d\to \RR$ and any $t\in [0,T]$ we have
\begin{align*}
\langle f, \cL(X_t)\rangle
& =\EE[f(\Phi(t,\xi;W))]
= \int_{C_T} \int_{\RR^d} f(\Phi(t,x;\omega)) \rho(x)\mathd x\, \mu(\dd \omega)\\
& = \int_{C_T} \int_{\RR^d} f(y) \rho(\Psi(t,x;\omega)) \det D\Psi(t,y;\omega) \dd y\, \mu( \dd \omega)\\
& = \int_{\RR^d} f(y) \bigg[ \int_{C_T} \rho(\Psi(t,y;\omega)) \det D\Psi(t,y;\omega) \, \mu( \dd \omega)\bigg]\, \dd y.
\end{align*}
As a consequence, $\cL(X_t)$ is absolutely continuous  w.r.t. Lebesgue with density given by
\[
\rho_t(x) := \int_{C_T} \rho(\Psi(t,y;\omega)) \det D\Psi(t,y;\omega) \, \mu( \dd \omega).
\]
Given the bounds on the Jacobian of $\Phi$ and $\Psi$ as above, for $p\in [1,\infty)$ we can estimate $\| \rho_t\|_{L^p}$ by
\begin{align*}
  \int_{\RR^d} |\rho_t(x)|^p \dd x
  & \leq \int_{\RR^d} \int_{C_T} |\rho(\Psi(t,x;\omega))|^p |\det D\Psi(t,y;\omega)|^p \, \mu( \dd \omega)\\
  & \lesssim_{\div b} \int_{C_T} \int_{\RR^d} |\rho(\Psi(t,x;\omega))|^p \dd x \, \mu( \dd \omega)\\
  & = \int_{C_T} \int_{\RR^d} |\rho(y)|^p \det D\Phi(t,y;\omega) \dd y \, \mu( \dd \omega)
  \lesssim_{\div b} \int_{\RR^d} |\rho(y)|^p \dd y;
\end{align*}
the estimate for $p=\infty$ is similar.
To prove the last statement, applying several times the fact that the incompressibility constants of $\Phi(t,\cdot;\omega)$ are controlled by $\exp(\|\div b\|_{L^1_TL^\infty_x})$ uniformly in $\omega\in C_T$ for any $t\in [0,T]$ and $A\subset\RR^d$, see see \eqref{eq:PhiIncompressBound}, we have that
\begin{align*}
  \cL(X_t)(A)
  \lesssim_h \int_A \int_{C_T} \rho(\Psi(t,y;\omega) \, \mu( \dd \omega) \dd y
  \lesssim_h \int_{C_T} \int \rho(x) \mathds{1}_{\psi(t,A;\omega)}(x) \dd x \, \mu( \dd \omega);
\end{align*}
the conclusion follows by observing that $\rho\in L^1_x$ is uniformly integrable and that $|\psi(t,A;\omega)|\lesssim_h |A|$ uniformly in $t,\omega$.
\end{proof}

With the a priori bound \eqref{eq:LipKernelApriori} at hand, we are now ready to prove existence of solutions to \eqref{eq:ConvMckeanSDE} for integrable drifts $b$, up to the price of requiring additional assumptions on $\div b$ and $\rho$.

\begin{prop}\label{prop:PathwiseExistence2}
Suppose $b \in L^1_T L^q_x$, $\div b \in L^1_T L^\infty_x$ and $\cL(\xi)(\dd x):=\rho(x)\dd x$ for a probability density $ \rho \in L^p_x$ with $p,q\in [1,\infty]$ such that
\begin{equation}\label{eq:exist-coeff-pq}
  \frac{1}{2q}+\frac{1}{p}\leq 1.
\end{equation}
Then there exists at least one weak solution $X_t$ to \eqref{eq:ConvMckeanSDE} in the sense of Definition \ref{def:WeakSolsII}, which moreover satisfies $\sup_{t\in [0,T]} \|\cL(X_t)\|_{L^p_x} <\infty$.
\end{prop}
\begin{proof}
The proof is based on classical compactness arguments; since $\rho\in L^1_x\cap L^p_x$, by interpolation w.l.o.g. we can assume equality holds in \eqref{eq:exist-coeff-pq}.

Let $\{b^n\}_n$ be a sequence in $L^1_T C^1_b$ satisfying 
$\|b^n\|_{L^1_T L^q_x}\leq \|b\|_{L^1_T L^q_x}$, $\|\div b^n\|_{L^1_T L^\infty_x}\leq \|\div b\|_{L^1_T L^\infty_x}$ (take for instance $b^n_s := b_s \ast \varphi_{1/n}$ with $\{\varphi_\eps\}_{\eps>0}$ standard mollifiers); by Section \ref{sec:LipWellPosed}, for each $n\geq 1$ there exists a unique strong solution $(X^n,\cL(X^n)) \in C_T\times \cP(C_T)$ to the DDSDE associated to $b^n$, with input data $(\xi,W)$. By Proposition \ref{prop:LawLpApriori}, we can find some $C=C(\|\div b\|_{L^1_T L^\infty_x})$ such that
\begin{equation*}
  \sup_{t \in [0,T]} \|\cL(X^n_t)\|_{L^p} \leq C \| \rho\|_{L^p_x} \quad \forall\, n\in\NN.
\end{equation*}

Set $Z_t^n := \int_0^{t} b^n_s\ast \cL(X^n_s)(X^n_s)\,\dd s = X^n_t-\xi-W_t$. Our first goal is to establish uniform bounds on $\{Z_n\}_n$. It holds that
\begin{align*}
  \EE[\| Z^n\|_{W^{1,1}_T}]
  & \leq \EE\bigg[\int_0^T |b_s^n|\ast \cL(X^n_s)(X^n_s)\, \dd s\bigg]\\
  & \leq \int_0^T \langle |b^n_s|\ast \cL(X^n_s),\cL(X^n_s)\rangle\, \dd s\\
  & \leq \| b^n\|_{L^1_T L^q_x} \sup_{t\in [0,T]} \|\cL(X^n_s)\|_{L^p_x}^2
  \lesssim \| b\|_{L^1_T L^q_x}\, \| \rho\|_{L^p_x}^2
\end{align*}
where in the third passage we applied H\"older's and Young's inequalities.

Then Helly's selection theorem, together with the fact that $W^{1,1}_T$ embeds compactly in $L^r_T$ for any $r<\infty$ as well as in the topology of pointwise convergence, implies that $\{Z^n\}_n$ is tight in $L^r_T$ and $\{(\xi,W,Z^n)\}_n$ is tight in $\RR^d\times C_T\times L^r_T$.

We can therefore apply Prokhorov's and Skorohod's theorems to construct another probability space $(\tilde{\Omega},\tilde{\cF},\tilde{\PP})$, carrying a family of random variables $(\tilde{\xi}^n, \tilde{W}^n,\tilde{Z}^n)$ converging $\tilde{\PP}$-a.s. to $(\tilde{\xi},\tilde{W},\tilde{Z})$ in $\RR^d\times C_T\times L^r_T$, such that $\cL_{\tilde{\PP}} (\tilde{\xi}^n, \tilde{W}^n,\tilde{Z}^n) = \cL_{\PP}(\xi,W,Z^n)$.
From this we deduce as before that $\tilde{\PP}$-a.s. $\tilde{Z}^n\to \tilde{Z}$ also in the sense of pointwise convergence, similarly for $\tilde{X}^n:=\tilde{\xi}^n+\tilde{W}^n+\tilde{Z}^n$;
moreover $\tilde{X}^n$ are still solutions to the DDSDE associated to $b^n$ and $\tilde{Z}^n_t = \int_0^t b^n_s \ast \cL_{\tilde{\PP}}(\tilde{X}^n_s) (\tilde{X}^n_s) \dd s$.

In order to conclude, it remains to show that $(\tilde{X},\tilde{W},\tilde{\xi})$ is a weak solution to the DDSDE associated to $b$, on the probability space $(\tilde{\Omega},\tilde{\cF},\tilde{\PP})$.
For notational simplicity we will drop the tildes from now on. 

We first treat the case $p>1$.
Since $\PP$-a.s. $X^n_t\to X^n$, $\cL(X^n_t)\rightharpoonup \cL(X_t)$ weakly and they are uniformly bounded in $L^p_x$ with $p>1$; thus $\cL(X_t)\in L^p_x$ and by lower semicontinuity of weak convergence (weak-$\ast$ if $p=\infty$) it holds that
\[
\sup_{t\in [0,T]} \| \cL(X_t)\|_{L^p_x} \leq \sup_{t\in [0,T]} \liminf_{n\to\infty} \| \cL(X^n_t)\|_{L^p_x} \leq C \| \rho\|_{L^p_x},
\]
which checks requirement $(1)$ of Definition \ref{def:WeakSolsII}.
The verification of $(2)$ is similar and we are left with showing that $Z^n\to Z$ for $Z_t = \int_0^t b_s\ast\cL(X_s)(X_s)\dd s$.

In the case $p>1$, $q<\infty$, we may further assume that $b^n\to b$ in $L^1_T L^q_x$. We fix $\eps>0$ and choose $\bar{b}\in C^1_b([0,T]\times \RR^d;\RR^d)$ such that $\|b-\bar{b}\|_{L^1_T L^q_x}<\varepsilon$ and split the estimate as
\begin{align*}
  \|Z^n - Z\|_{C_T}
  & \leq \int_0^T |b^n_s-\bar{b}_s|\ast \cL(X^n_s) (X^n_s) \dd s + \int_0^T |\bar{b}_s\ast \cL(X^n_s) (X^n_s)-\bar{b}_s\ast \cL(X_s) (X_s)| \dd s\\
  & \quad + \int_0^T |\bar{b}_s - b_s|\ast \cL(X_s) (X_s) \dd s\\
  & := I^n_1 + I^n_2 + I_3.
\end{align*}
 Regarding $I^n_1$, we can apply the uniform estimates on $\| \cL(X^n_t)\|_{L^p_x}$ to obtain
\begin{align*}
  \EE[I^n_1]
  \lesssim \sup_{s\in [0,T]} \| \cL(X^n_s)\|_{L^p_x} ^2 \| b^n-\bar{b}\|_{L^1_T L^q_x}
  \lesssim \| b^n-b\|_{L^1_T L^q_x} + \| b-\bar{b}\|_{L^1_T L^q_x};
\end{align*}
similarly for $I_3$.
On the other hand, since $\bar{b}\in C^1_b$ and $\cL(X^n)\rightharpoonup \cL(X)$, $\bar{b}\ast \cL(X^n)$ converges uniformly to $\bar{b}\ast \cL(X)$; combined with $X^n\to X$ $\PP$-a.s. and the uniform bounds, by dominated convergence we have $\EE[I^n_2]\to 0$. Thus overall we obtain
\begin{align*}
  \limsup_{n\to\infty} \EE[\| Z^n-Z\|_{C_T}]
  & \leq \limsup_{n\to\infty} (\EE[I^n_1] + \EE[I^n_2] + \EE[I_3])\\
  & \lesssim \limsup_{n\to\infty} \| b^n-b\|_{L^1_T L^q_x} + 2 \| b-\bar{b}\|_{L^1_T L^q_x}
  \lesssim 2\eps;
\end{align*}
by the arbitrariness of $\eps>0$ we deduce that $\EE[\| Z^n-Z\|_{C^T}] \to 0$ and that $X$ is a weak solution.

In the case $p=1, q=\infty$ we cannot assume $\| b^n-b\|_{L^1_T L^\infty_x}\rightarrow 0$, but only that $b^n_t(x)\to b_t(x)$ for Lebesgue a.e. $(t,x)\in [0,T]\times\RR^d$.
Since $X^n_t\rightharpoonup X_t$, the sequence $\{\cL(X^n_t)\}_n$ is tight, which together with \eqref{eq:apriori-equiint} implies its equiintegrability.
But then by Dunford--Pettis theorem $\cL(X_t)\in L^1_x$ and $\cL(X^n_t)\rightharpoonup \cL(X_t)$ weakly in $L^1_x$, checking requirement (1) of Definition \ref{def:WeakSolsII}; (2) follows trivially since $b\in L^1_T L^\infty_x$.
Finally, since $\cL(X^n_t)\rightharpoonup \cL(X_t)$ weakly in $L^1_x$, we can apply Lemma \ref{lem:weak-conv-L1} from Appendix \ref{app:maximal} at any fixed $t\in [0,T]$ to conclude that $\EE[\| Z^n-Z\|_{C_T}] \to 0$ still holds.
\end{proof}

Having established existence under suitable integrability assumptions on $b$, it is natural to establish uniqueness under similar requirements on the derivative $D b$. We start with a conditional result of uniqueness in a class of sufficiently regular solutions.

\begin{prop}\label{thm:MKV-uniqueness-under-Sobolev-bis}
Let $q,\,p \in (1,\infty]$ be such that $\frac{1}{p}+\frac{1}{q}<1$ and assume $b \in L^1_T W^{1,q}_x$, $\rho \in L^p_x$. Then both uniqueness in law and pathwise uniqueness holds for \eqref{eq:ConvMckeanSDE} in the class of solutions satisfying
\begin{equation}\label{eq:uniqueness-class}
  \sup_{t\in [0,T]} \| \cL(X_t)\|_{L^p_x} <\infty
\end{equation}
in the following sense: if $X^1$ and $X^2$ are both solutions satisfying \eqref{eq:uniqueness-class}, then their laws as probabilities on the path space $C_T$ coincide; if $X^i$ are defined on the same probability space and solve the DDSDE w.r.t. same input data $(\xi,W)$, then $X^1=X^2$ $\PP$-a.s.
\end{prop}
\begin{proof}
Given two solutions $X^i$ as above, let us set $\mu^i_t = \cL(X^i_t)$ and $b^i_t(x) = (b_t\ast \mu^i_t)(x)$; from Proposition \ref{prop:PathwiseExistence2}, for all $t\in [0,T]$, $\mu^i_t \in L^1_x\cap L^p_x$ for $p>q'$, by interpolation $\mu^i \in L^\infty_t L^{q'}_x$ and so by Young's inequality
\begin{equation*}
  \| b^i\|_{L^1_T W^{1,\infty}_x} \lesssim \| b\|_{L^1_T W^{1,q}} \| \mu^i\|_{L^\infty_T L^{q'}_x} \lesssim 1.
\end{equation*}
In particular, both drifts $b^i$ are spatially Lipschitz; moreover by arguing as in Proposition \ref{prop:baby-yamada-watanabe}, we only need to check pathwise uniqueness for solutions $X^i$ to \eqref{eq:ConvMckeanSDE} satisfying \eqref{eq:uniqueness-class}.

First observe that the difference $Y:=X^1-X^2$ satisfies
\begin{align*}
  |Y_t| & \leq \int_0^t |b^1_s(X^1_s) - b^1_s(X^2_s)|\dd s + \int_0^t |(b^1-b^2)_s(X^2_s)|\dd s\\
  & \leq \int_0^t \| b_s\ast \mu^1_s\|_{W^{1,\infty}_x} |Y_s|\dd s + \int_0^t \| b_s\ast (\mu^1_s-\mu^2_s)\|_{L^\infty_x} \dd s\\
  & \lesssim_{\mu^i} \int_0^t \| b_s\|_{W^{1,q}_x} (1+|Y_s|)\dd s;
\end{align*}
applying Gr\"onwall, we find a deterministic constant $C=C(b, \mu^i)>0$ such that $\| X^1-X^2\|_{C_T} = \| Y\|_{C_T} \leq C$ $\PP$-a.s.

As a consequence, we also have that $d_{r'} (\mu^1_t, \mu^2_t)\leq C< \infty$ for any $r'\in (1,\infty)$; in particular, the assumptions of Lemma \ref{lem: maximal function1} from Appendix \ref{app:maximal} are met, as there exists $r>1$ such that $r/q+1/p\leq 1$. It follows that
\begin{align*}
  \| b_s\ast (\mu^1_s-\mu^2_s)\|_{L^\infty_x}
  \lesssim \| b_s\|_{W^{1,q}} 
  \big(\|\mu^1_s\|_{L^p_x}^{1/r} + \|\mu^2_s\|_{L^p_x}^{1/r}\big) d_{r'}(\mu^1_s,\mu^2_s) \lesssim_{\mu^i} \| b_s\|_{W^{1,q}} \| Y_s\|_{L^{r'}_\Omega}.
\end{align*}
We use this new bound to improve the estimates on $Y$, giving:
\[
|Y_t|\lesssim_{\mu^i} \int_0^t \| b_s\|_{W^{1,q}_x} (|Y_s| + \| Y_s\|_{L^{r'}_\Omega}) \dd s
\]
and now applying Minkowski's integral inequality (with $r'>1$) we obtain
\[
\| Y_t\|_{L^{r'}_\Omega} \lesssim_{\mu^i} \int_0^t \| b_s\|_{W^{1,q}_x}\, \| Y_s\|_{L^{r'}_\Omega}\, \dd s.
\]
However, by the hypothesis $b\in L^1_T W^{1,q}_x$ and applying Gr\"onwall's lemma, we deduce $\| Y_t\|_{L^{r'}_\Omega} = 0$ for all $t\in [0,T]$ and so $X^1=X^2$ $\PP$-a.s.
\end{proof}

Combining Propositions \ref{prop:PathwiseExistence2} and \ref{thm:MKV-uniqueness-under-Sobolev-bis} we get the following existence and uniqueness result.

\begin{thm}\label{thm:exist-uniq-DDSDE}
Let $q,\,p \in (1,\infty]$ be such that $\frac{1}{p}+\frac{1}{q}<1$ and $b \in L^1_T W^{1,q}_x$ with $\div b \in L^1_T L^\infty_x$.
Then for any $\rho \in L^p_x$ there exists a strong solution $X_t$ satisfying $\sup_{t\in [0,T]} \| \cL(X_t)\|_{L^p_x}<\infty$, which is unique in this class (in the sense of Proposition \ref{thm:MKV-uniqueness-under-Sobolev-bis}).
\end{thm}
\begin{proof}
The assumption on $(p,q)$ implies that \eqref{eq:exist-coeff-pq} holds, thus by Proposition \ref{prop:PathwiseExistence2} weak existence holds; as observed in the proof of Proposition \ref{thm:MKV-uniqueness-under-Sobolev-bis}, the drift $\bar{b}_t=b_t\ast \cL(X_t)\in L^1_T W^{1,\infty}_x$ and so arguing as in Proposition \ref{prop:equivalence-weak-strong} we deduce strong existence. Uniqueness follows again from Proposition \ref{thm:MKV-uniqueness-under-Sobolev-bis}.
\end{proof}
Theorem \ref{thm:exist-uniq-DDSDE}, together with the a priori estimates from Proposition \ref{prop:LawLpApriori}, guarantees that such solution $X$ is the only possible accumulation point of solutions $X^n$ constructed from mollified drifts $b^n=b\ast\varphi_{1/n}$.
It doesn't however exclude the existence of different solutions $\tilde X$, associated to the same input data $(\xi,W)$, which fail to meet the requirement $\sup_{t\in [0,T]} \| \cL(\tilde{X}_t)\|_{L^p_x}<\infty$.

In order to strengthen the result and truly establish uniqueness in law, without any regularity assumption on $\cL(\tilde{X})_t$, we need more restrictive conditions on $b$; the advantage is that we deduce some strong stability estimates, in the form of equation \eqref{eq:stability-sobolev-DDSDE} below. This result relies on the results of \cite{caravenna2021directional} and in particular Lemma \ref{lem:maximal-function2}. 
\begin{prop}\label{prop:stability-sobolev-DDSDE}
Let $(p,q,r)\in (1,\infty)^3$, $b\in L^1_T W^{1,q}_x$ with $\div b\in L^1_T L^\infty_x$ and
\[
q>d,\quad \frac{r}{p}+\frac{1}{q}\leq 1,\quad \frac{1}{r}+\frac{1}{r'}=1.
\]
Also assume we are given input data $\cL(\xi^1,W^1)=\mu^{\xi^1}\otimes \mu^{W^1}\in \cP_{r'}(\RR^d\times C_T)$, such that $\mu^{\xi^1}(\dd x)=\rho(x)\dd x$ for some $\rho\in L^q_x$. Denote by $\mu^{X^1}$ the unique solution associated to $\mu^{\xi^1+W^1}$, in the sense of Theorem \ref{thm:exist-uniq-DDSDE}.
Then there exists a constant $C$, depending on $d,T,p,q,r$, $\| \div b\|_{L^1_T L^\infty_x}$, $\| b\|_{L^1_T W^{1,p}_x}$ and $\| \rho\|_{L^q}$, with the following property: for any other solution $\mu^{X^2}$ to \eqref{eq:ConvMckeanSDE} with input $\mu^{Y^2}$,
\begin{equation}\label{eq:stability-sobolev-DDSDE}
  d_{r'}(\mu^{X^1},\mu^{X^2}) \leq C d_{r'}\big(\mu^{\xi^1+W^1},\mu^{Y^2}\big).
\end{equation}
\end{prop}
\begin{proof}
Since $q>d$, by Sobolev embeddings $b\in L^1_T C^0_b$, so that the DDSDE \eqref{eq:ConvMckeanSDE} is meaningful in the sense of Definition \ref{def:WeakSolsI}, without the need of any regularity assumption on $\mu^{X^2}$.
By Remark \ref{rem:baby-yamada-watanabe}, given any coupling of $\mu^{\xi^1+W^1}, \mu^{Y^2}$, we can construct a coupling of the associated solutions $\mu^{X^1}, \mu^{X^2}$ as well, so from now on we will work with $X^i,\xi^1,W^1,Y^2$ all defined on the same probability space; also for simplicity we will adopt the compact notation $Y^1=\xi^1+W^1$.
Since $X^i$ are solutions to the DDSDE and $b\in L^1_T C^0_b$, we have the trivial estimate
\begin{equation*}
  \EE\big[\| X^i\|_{C_T}^{r'}\big] \lesssim \| b\|_{L^1_T W^{1,q}_x}^{r'} + \EE\big[\| Y^i\|_{C_T}^{r'}\big]
\end{equation*}
which shows that $\mu^{X^i}\in \cP_{r'}(C_T)$.

Having made these preparations, we can now pass to the proof of estimate \eqref{eq:stability-sobolev-DDSDE}. Set $\mu^i_t = \cL(X^i_t)$, then the process $X^1-X^2$ satisfies
\[
|X^1_t - X^2_t|\leq \int_0^t |b_s\ast \mu^1_s (X^1_s) - b_s\ast \mu^2_s (X^2_s)|\dd s + | Y^1_t-Y^2_t|;
\]
since $\sup_{t\in [0,T]} \| \mu^1_t\|_{L^p_x} < \infty$, we can apply Lemma \ref{lem:maximal-function2} from Appendix \ref{app:maximal} to obtain
\begin{align*}
  |X^1_t - X^2_t|
  & \lesssim \int_0^t \|b_s\|_{W^{1,q}_x} (1+ \| \mu^1_s\|_{L^p_x})\, (|X^1_s-X^2_s| + d_{r'} (\mu^1_s,\mu^2_s))\,\dd s + | Y^1_t-Y^2_t|\\
  & \lesssim_{\div b, \rho} \int_0^t \|b_s\|_{W^{1,q}_x}\, (|X^1_s-X^2_s| + d_{r'} (\mu^1_s,\mu^2_s))\,\dd s + \| Y^1-Y^2\|_{C_T};
\end{align*}
taking the $L^{r'}_{\Omega}$-norm on both sides and applying Minkowski's inequality we find
\[
\| X^1_t-X^2_t\|_{L^{r'}_\Omega} \lesssim \int_0^t \| b_s\|_{W^{1,q}_x} \| X^1_s-X^2_s\|_{L^{r'}_\Omega} \dd s + \|\| Y^1-Y^2\|_{C_T}\|_{L^{r'}_\Omega}.
\]
By Gr\"onwall's lemma we can find some constant $\kappa$ such that
\[
\sup_{t\in [0,T]} d_{r'}(\mu^1_t,\mu^2_t) \lesssim e^{\kappa \| b\|_{L^1_T W^{1,q}_x}} \|\| Y^1-Y^2\|_{C_T}\|_{L^{r'}_\Omega}.
\]
In order to obtain a bound in the path space $C_T$ it suffices to observe that the previous inequalities also imply, again by Gr\"onwall,
\[
\| X^1-X^2\|_{C_T} \lesssim_b \| b\|_{L^1_T W^{1,q}_x} \sup_{t\in [0,T]} d_{r'}(\mu^1_t,\mu^2_t) + \| Y^1-Y^2\|_{C_T};
\]
inserting the estimate for $d_{r'}(\mu^1_t,\mu^2_t)$ and taking expectation, we find a constant $C$ as above s.t.
\begin{equation*}
  d_{r'}(\mu^{X^1},\mu^{X^2})
  \leq \EE[\| X^1-X^2\|_{C_T}^{r'}]^{1/r'}
  \leq C\, \EE[\| Y^1-Y^2\|_{C_T}^{r'}]^{1/r'}.
\end{equation*}
Minimizing over all possible couplings $(Y^1,Y^2)$ given the conclusion.
\end{proof}
\section{Mean Field Convergence}\label{sec:MFL}
In this section we prove mean field limits for the interacting particle systems associated to the drifts considered in Sections \ref{sec:BenchmarkResults}-\ref{sec:convol}. We start by presenting Tanaka's idea in a rather abstract fashion and then apply it to our cases of interest.

\subsection{Abstract criteria}

Let us shortly recall the setup. Given a family of $C_T$-valued random variables $\{Y^i,i\in\NN\}$ defined on a probability space $(\Omega,\cF,\PP)$ and a measurable drift $B:[0,T]\times \RR^d\times \cP(\RR^d)\to\RR$ (possibly with $\cP(\RR^d)$ replaced by $\cP_p(\RR^d)$ for some $p\in [1,\infty)$), we say that a family $X^{(N)}:=\{X^{i,N}\}_{i=1}^N$, seen as a $(C_T)^N$-valued random variable, is a solution to the $N$-particle system if for $\PP$-a.e. $\omega\in \Omega$ it holds
\begin{equation}\label{eq:particle-system}
  X^{i,N}_t(\omega) = \int_0^t B_s(X^{i,N}_s(\omega), L^N(X^{(N)}_s(\omega)))\,\dd s + Y^i_t\quad \forall\, t\in [0,T],\, i=1,\ldots, N
\end{equation}
where the integral is imposed to be meaningful in the Lebesgue sense for any such $\omega$ and
\[
L^N(X^{(N)}_t(\omega)): = \frac{1}{N}\sum_{i=1}^N \delta_{X^{i,N}_t (\omega)}
\]
denotes the empirical measure of the system at time $t$. For notational simplicity we will assume different particle systems to be defined on the same probability space, although we could allow $(\Omega,\cF,\PP)$ to vary as a function of $N$. In this case we would adopt the notation $Y^{(N)}:=\{Y^{i,N}\}^N_{i=1}$ for the sequence of input data below.

We will adopt the following notion of mean field convergence.

\begin{defn}[Mean Field Convergence]\label{defn:MFL} Let $\{Y^i\}_{i\in\NN}$ be a sequence of input data such that, defining $Y^{(N)}= \{Y^i\}_{i=1}^N$, one has $\PP$-a.s.
\[
L^N(Y^{(N)}(\omega)) = \frac{1}{N}\sum_{i=1}^N \delta_{Y^i(\omega)} \rightharpoonup \mu^Y \text{ in } \cP(C_T) \text{ as } N\to\infty,
\]
with $Y$ such that there exists a unique solution, $\mu^X$ to the DDSDE \eqref{eq:GenMckeanSDE} with input $\mu^Y$.

Then we say that mean field convergence of the particle system \eqref{eq:intro-particle-system} to the DDSDE \eqref{eq:introddsde} holds along $\{Y^i\}_{i\in\NN}$ if for any sequence of solutions $\{X^{i,N}\}$ to the associated $N$-particle, $\PP$-a.s. we have
\[
L^N(X^{(N)}(\omega))\rightharpoonup \mu^X \text{ in } \cP(C_T) \text{ as } N\to\infty.
\]
\end{defn}
Let us remark that although our definition is non-standard, it is well suited to our purposes below: it doesn't require the solution $\{X^{i,N}\}$ to be unique, only the candidate limit $\mu^X$; similarly we don't impose the property to hold for all possible families $\{Y^i\}$, allowing us some freedom in their choice. As the applications in the next section will show, it covers the i.i.d. case for a large class of $\cL(Y)\in \cP(C_T)$.

We are now ready to present Tanaka's argument; as already mentioned in the introduction, it can be regarded as a ``transfer principle'' from the DDSDE to the particle system.

\begin{prop}\label{prop:abstract-tanaka}
Consider the DDSDE \eqref{eq:GenMckeanSDE} for a given drift $B$. Assume there exists $\mu^{Y^1}\in \cP(\RR^d)$ with the following properties:
\begin{enumerate}[label=\roman*.]
  \item there exists a unique solution $\mu^{X^1}$ associated to $\mu^{Y^1}$;
  \item there exist $p\in [1,\infty)$, a measurable function $F:\cP_p(C_T)\times \cP(C_T)\to [0,+\infty]$ and an increasing continuous function $M:[0,+\infty)\to [0,+\infty)$ with $M(0)=0$ such that, for any other input $\mu^{Y^2}$ and any solution $\mu^{X^2}$ associated to $\mu^{Y^2}$, it holds that
\begin{equation}\label{eq:tanaka-assumption}
  d_p(\mu^{X^1},\mu^{X^2}) \leq F(\mu^{Y^1},\mu^{Y^2})\, M(d_p(\mu^{Y^1},\mu^{Y^2})).
\end{equation}
\end{enumerate}
Then for any solution $\{X^{i,N}\}_i$ to the $N$-particle system associated to $\{Y^i\}_i$, $\PP$-a.s. 
\begin{equation}\label{eq:tanaka-conclusion}
  d_p(\mu^{X^1},L^N(X^{(N)}(\omega)) \leq F(\mu^{Y^1},L^N(Y^{(N)}(\omega)))\, M\left(d_p\left(\mu^{Y^1},L^N(Y^{(N)}(\omega))\right)\right).
\end{equation}
\end{prop}

Before giving the proof, let us explain the meaning of condition \eqref{eq:tanaka-assumption}. In the simplest cases like Section \ref{sec:LipWellPosed}, it can be shown that the DDSDE \eqref{eq:GenMckeanSDE} is well-posed for all $\mu^Y\in \cP(C_T)$ and defines a continuous solution map $S^B:\cP(C_T)\to \cP(C_T)$ by $\mu^Y\mapsto \mu^X$. In this case condition \eqref{eq:tanaka-assumption} is satisfied for $F\equiv 1$ and $M$ the modulus of continuity of $S^B$ around $\mu^{Y^1}$; if $S^B$ is globally Lipschitz, then $M(r) = C r$ for some constant $C>0$. The assumptions of Proposition \ref{prop:abstract-tanaka} are just a convenient generalization to the case in which $S^B$ is not well-defined anymore, or maybe not on an open neighbourhood of $\mu^{Y^1}$ in $\cP_p(C_T)$ but only on a strict subset of it (encoded by the presence of $F$ which can take value $+\infty$).

\begin{proof}
Let the family $Y^{(N)}=\{Y^i\}_{i=1}^N$ be defined on a probability space $(\Omega,\cF,\PP)$, together with a solution $X^{(N)}=\{X^{i,N}\}_{i=1}^N$ to the $N$-particle system; we can regard $X^{(N)},\,Y^{(N)}$ as $(C_T)^N$-valued random variables.
Define a new probability space $(\Omega_N,\cF_N,\PP_N)$ by
\begin{align*}
  \Omega_N := \{1,\ldots,N\},\quad \cF_N:= 2^{\Omega_N},\quad \PP_N := \frac{1}{N}\sum_{i=1}^N\delta_i.
\end{align*}
We can identify any element $y^{(N)} = \{y^i\}_{i=1}^N\in (C_T)^N$ as a random variable $\Omega_N \ni i \mapsto y^i \in C_T$. With this identification in place, for any $\varphi \in C_b(C_T)$ we have
\begin{equation*}
  \EE_{\PP_N} \Big[\varphi\big(y^{(N)}\big)\Big] = \frac{1}{N}\sum_{i=1}^N \varphi(y^i) = \Big\langle \varphi, \frac{1}{N} \sum_{i=1}^N \delta_{y^i} \Big\rangle = \langle \varphi, L^{N}(y^{(N)})\rangle;
\end{equation*}
that is, $L^N(y^{(N)}) = \cL_{\PP_N} (y^{(N)})$. Applying the above to $Y^{(N)}(\omega)$, $X^{(N)}(\omega)$ we deduce that for $\PP$-a.e. $\omega$,
\begin{equation*}
  X^{i,N}_t(\omega) = \int_0^t B_s(X^{i,N}_s(\omega), \cL_{\PP_N}(X^{(N)}_s(\omega))\,\dd s + Y^i_t(\omega) \quad \forall\, i\in\Omega_N.
\end{equation*}
Namely for any such $\omega$ the solution $X^{(N)}(\omega)$ to the particle system can be regarded as a solution to the DDSDE \eqref{eq:GenMckeanSDE} with input $Y^{(N)}(\omega)$ with respect to the probability space $(\Omega_N,\cF_N,\PP_N)$.
Applying assumption \eqref{eq:tanaka-assumption} for the choice $\mu^{Y^2}= \cL_{\PP_N}(Y^{(N)}(\omega))=L^N(Y^{N}(\omega))$, similarly $\mu^{X^2}$, readily gives the conclusion.
\end{proof}

\begin{rem}\label{rem:particle}
The proof also reveals the following simple fact: if the DDSDE \eqref{eq:GenMckeanSDE} is well-posed with a globally defined solution map $S^B:\cP_p(C_T)\to\cP_p(C_T)$, then any solution to the particle system \eqref{eq:particle-system} satisfies $X^{(N)}(\omega)=S^B(Y^{(N)}(\omega))$. In particular strong existence, pathwise uniqueness and uniqueness in law hold for the particle system as well.
\end{rem}

We immediately deduce the following.

\begin{cor}\label{cor:tanaka-suff}
Suppose $\mu^{Y^1}$ satisfies the assumptions of Proposition \ref{prop:abstract-tanaka}. Then the mean field convergence holds along any sequence $\{Y^i\}_i$ such that $\PP$-a.s.,
\begin{equation*}
  \limsup_{N\to\infty} F\left(\mu^{Y^1},L^N(Y^{(N)}(\omega))\right)<\infty,
  \quad
  \lim_{N\to\infty} d_p\left(\mu^{Y^1},L^N(Y^{(N)}(\omega))\right) = 0.
\end{equation*}
\end{cor}
The above sufficient conditions still need to be checked case-by-case; however if $F\equiv 1$, the next general lemma shows that they hold in the classical case of i.i.d. data.
\begin{lem}\label{lem:IIDConverge}
Let $E$ be a separable Banach space, $p\in [1,\infty)$ and $\{Z^i\}_{i\geq 1}$ be a family of i.i.d. random variables with common law $\mu\in \cP_p(E)$. Then
\begin{equation}
  \lim_{N\rightarrow \infty}d_p(L^N(Z^{(N)}(\omega)),\mu ) =0\quad \text{for $\PP$-a.e. }\omega.
\end{equation}
\end{lem}
\begin{proof}
The statement is very close to \cite[Lem. 54]{coghi2020pathwise}; the fact that $L^N(X^N(\omega))\rightharpoonup \mu$ for $\PP$-a.e. $\omega$ is a classical result, sometimes referred to as Glivenko--Cantelli theorem. Moreover by the law of large numbers we have
\begin{align*}
  \| L^N(X^{(N)}(\omega))\|_p^p
  = \frac{1}{N} \sum_{i=1}^N \|X^i(\omega)\|_E^p
  \to \int_E \| x\|_E^p\, \mu(\dd x)
  = \| \mu\|_p^p \quad \PP\text{-a.s.}
\end{align*}
and combining these facts gives the conclusion.
\end{proof}


\begin{rem}
If the data $\{Y^i\}_{i=1}^N$ are exchangeable (which is true in the i.i.d. case), then so are $\{X^{i,N}\}$. It then follows from \cite[Proposition 2.2]{sznitman1991topics} that the mean field convergence is equivalent to the property of propagation of chaos; the proof therein is based on purely combinatorial arguments and does not rely on any assumption on $Y$, thus can be applied directly to our setting.
\end{rem}

\begin{rem}
Often one is not just interested in establishing the mean field limit property but also to derive rates of convergence for $d_p(L^N(X^{(N)}), \cL(X))$ as a function of $N$; thanks to Proposition \ref{prop:abstract-tanaka}, once the functions $F$ and $M$ are explicit, it suffices to find similar rates for $d_p(L^N(Y^{(N)}), \cL(Y))$. In the case of i.i.d. data, this problem has been studied extensively, see for instance \cite{bolley2007quantitative, dereich2013constructive}; the most advanced result in the case of measures supported on $\RR^d$ we are aware of are the ones presented in \cite{fournier2015rate}. However rates of convergence in $\cP_p(C_T)$ are much more difficult to obtain and only partial results (mostly relying on Sanov's theorem and Talagrand inequalities) are available, see \cite{bolley2010quantitative, boissard2011simple}.
\end{rem}

\subsection{Applications to particular drifts}

We are now ready to apply the previous results to the drifts considered in Sections \ref{sec:BenchmarkResults}-\ref{sec:convol}.

\subsubsection{Osgood drifts}

We assume that $B$ satisfies Assumption \ref{ass:GenLinearGrowth}, with $p=1$ and Assumption \ref{ass:Osgood}; it then follows from Theorem \ref{thm:osgood-stability} that the DDSDE has an associated solution map $S^B:\cP_1(C_T)\to \cP_1(C_T)$ with modulus of continuity $M$ of the form
\[
M(r) = G^{-1}\big( G(r) + \|h\|_{L^1_T}\big)
\]
for $G$ given by \eqref{eq:bihari-G}. By Remark \ref{rem:particle}, the particle system \eqref{eq:particle-system} is well-posed and $X^{(N)}(\omega) = S^B(Y^{(N)}(\omega))$.

\begin{cor}\label{cor:MFL-osgood}
Under the above assumptions on $B$, for any $\mu^Y\in \cP_1(C_T)$ we have 
\begin{equation*}
  d_1(L^N(X^{(N)}(\omega)), S^B(\mu^{Y})) \leq M(d_1(L^N(Y^{(N)}(\omega)), \mu^{Y})) \quad \PP\text{-a.s.};
\end{equation*}
as a consequence, mean field convergence holds along any sequence $\{Y^i\}$ such that
\begin{align*}
  d_1(L^N(Y^{(N)}(\omega)), \mu^Y)\to 0 \quad \PP\text{-a.s.}
\end{align*}
which is true in particular if $\{Y^i\}$ are taken as i.i.d. variables distributed as $\mu^Y$.
\end{cor}

\begin{proof}
The first statement is an immediate consequence of Proposition \ref{prop:abstract-tanaka} for the choice $F\equiv 1$ on $\cP_1(C_T)$and $M$ as above; the second one comes from Corollary \ref{cor:tanaka-suff}. The claim for i.i.d. variables follows from Lemma \ref{lem:IIDConverge}.
\end{proof}

Let us finally mention that the function $M$ can be explicitly computed from the modulus of continuity $f$, so one can also obtain more precise information on the rate of convergence.

\subsubsection{Monotone drifts}

We now assume that $B$ satisfies Assumption \ref{ass:GenLinearGrowth} and one between Assumptions \ref{ass:monotone-coeff}, \eqref{eq:monotone-coeff-2} or \eqref{eq:monotone-integr-cond} (for the same $p$); thus Lemma \ref{lem:monotone-sol-map} is applicable and we have a continuous solution map $S^B:\cP_p(C_T)\to \cP_p(C_T)$ associated to the DDSDE. As before, by Remark \ref{rem:particle} the particle system \eqref{eq:particle-system} is also well-posed.

\begin{cor}\label{cor:MFL-monotone}
Under the above assumptions on $B$, for any $\mu^Y\in \cP_p(C_T)$ we have 
\begin{equation*}
  d_p(L^N(X^{(N)}(\omega)), S^B(\mu^{Y})) \leq M^Y(d_p(L^N(Y^{(N)}(\omega)), \mu^{Y})) \quad \PP\text{-a.s.}
\end{equation*}
once we take $M^Y$ to be the modulus of continuity of $S^B$ around $\mu^Y$, namely
\begin{equation*}
  M(r) := \sup\Big\{ d_p(S^B(\mu^Y),S^B(\nu))\ \big\vert \ \nu\in \cP_p(C_T),\, d_p(\nu,\mu^{Y^1})\leq r\Big\}.
\end{equation*}
As a consequence, mean field convergence holds along any sequence $\{Y^i\}$ such that
\begin{align*}
  d_p(L^N(Y^{(N)}(\omega)), \mu^Y)\to 0 \quad \PP\text{-a.s.}
\end{align*}
which is true in particular if $\{Y^i\}$ are taken as i.i.d. variables distributed as $\mu^Y\in \cP_p(C_T)$.
\end{cor}

The proof is identical to that of Corollary \ref{cor:MFL-osgood} and thus omitted; let us point out however that in this case only qualitative continuity of $S^B$ is available, thus we can't obtain quantitative rates of convergence.

\subsubsection{Locally Lipschitz drifts}

Here we impose $B$ to satisfy Assumption \ref{ass:local-lipschitz}; by Proposition \ref{prop:local-lipschitz-stability}, we can talk of stability of solutions, but technically we don't have a well defined solution map on all of $\cP_p(C_T)$ since the result only applies to $\mu^Y$ belonging to the class
\begin{align*}
  \cE^\alpha := \Big\{ \nu\in \cP(C_T) : \int_{C_T} e^{\lambda \| \omega\|_{C_T}^\alpha }\, \nu(\dd \omega)<\infty \text{ for all }\lambda\in\RR\Big\}.
\end{align*}
Nevertheless, this class always includes convex combinations of Diracs and so by Remark \ref{rem:particle} the result transfers to well-posedness of the particle system \eqref{eq:particle-system}.

\begin{cor}\label{cor:MFL-local-lipschitz}
Let $B$ satisfy the above assumptions and fix $p,q\in [1,\infty)$ with $p<q$. Then there exists a constant $C$ such that for any $\mu^Y\in \cE^\alpha$ we have
\begin{equation*}
  d_p(L^N(X^{(N)}(\omega)), \mu^X)) \leq C \, F(L^N(Y^{(N)}(\omega)), \mu^Y)\, d_q(L^N(Y^{(N)}(\omega)), \mu^{Y}) \quad \PP\text{-a.s.}
\end{equation*}
for $F$ given by
\begin{equation*}
  F(\nu^1,\nu^2) := \int_{C_T} e^{C \| \omega\|_{C_T}^\alpha}\, \nu^1(\dd \omega) + \int_{C_T} e^{C \| \omega\|_{C_T}^\alpha}\, \nu^2(\dd \omega).
\end{equation*}
As a consequence, mean field convergence holds along any sequence $\{Y^i\}$ such that
\begin{align*}
  \limsup_{N\to\infty} F(L^N(Y^{(N)}(\omega)),\mu^Y)<\infty, \quad d_q(L^N(Y^{(N)}(\omega)), \mu^Y)\to 0 \quad \PP\text{-a.s.}
\end{align*}
which is true in particular if $\{Y^i\}$ are taken as i.i.d. variables distributed as $\mu^Y\in \cE^\alpha$.
\end{cor}

\begin{proof}
The first statement is a consequence of Propositions \ref{prop:local-lipschitz-stability} and \ref{prop:abstract-tanaka}; technically here contrary to \eqref{eq:tanaka-assumption} we have $p<q$, but the proof can be readapted to cover this case.
The second statement is analogous to Corollary \ref{cor:tanaka-suff} and so it only remains to cover the i.i.d. claim.
Convergence of $L^N(Y^{(N)})$ to $\mu^Y$ in $\cP_q$ follows from Lemma \ref{lem:IIDConverge}; we have
\begin{align*}
  F(L^N(Y^{(N)}(\omega)), \mu^{Y^1}) = \frac{1}{N} \sum_{i=1}^N e^{C \| Y^i\|_{C_T}^\alpha} + \int_{C_T} e^{C \| \omega\|_{C_T}^\alpha} \, \mu^{Y^1}(\dd \omega)
\end{align*}
and by SLLN it holds
\begin{align*}
  \lim_{N\to\infty} \frac{1}{N} \sum_{i=1}^N e^{C \| Y^i\|_{C_T}^\alpha} = \int_{C_T} e^{C \| \omega\|_{C_T}^\alpha} \, \mu^{Y^1}(\dd \omega)\quad \PP\text{-a.s.}
\end{align*}
implying the conclusion.
\end{proof}

\subsubsection{Convolutional drifts}

Let us now assume $B=b\ast\mu$ for $b$ satisfying the assumptions of Proposition \ref{prop:stability-sobolev-DDSDE}; recall in particular that $q>d$ and $L^1_T W^{1,q}_x \hookrightarrow L^1_T C^\alpha_x$ for some $\alpha >0$. As a consequence, the equation the particle system is always meaningful and existence of weak solutions to \eqref{eq:particle-system} holds by standard arguments; uniqueness however does not need to hold.

Recall that here the input is of the form $Y=\xi+W$, for independent $\xi,W$ with $\cL(\xi)\in L^p_x$; if we also considered $Y^{(N)}$ to be of the form $Y^i=\xi^i+W^i$, one would then invoke the more refined results from \cite{caravenna2021directional} to show that there exists a unique solution to the particle system \eqref{eq:particle-system} for $\PP$-a.e. realisation of $W^{(N)}(\omega)$ and Lebesgue-a.e. choice of initial data $\xi^{(N)}\in \RR^{Nd}$. However we will not need this fact here and we will therefore not impose such an assumption on $Y^{(N)}$.

\begin{cor}\label{cor:MFL-convol}
Let $b,\,\rho$ satisfy the assumptions of Proposition \ref{prop:stability-sobolev-DDSDE} and denote by $\mu^X$ the unique solution to the DDSDE \eqref{eq:ConvMckeanSDE} associated to $\mu^Y=\mu^{\xi+W}$. Then for any input $Y^{(N)}$ and any solution $X^{(N)}$ to the associated particle system \eqref{eq:particle-system} it holds that
\begin{equation*}
  d_{r'}(L^N(X^{(N)}(\omega)), \mu^{X^1}) \leq C d_{r'}(L^N(Y^{(N)}(\omega)), \mu^{Y}))\quad \PP\text{-a.s.}
\end{equation*}
for the parameter $r'\in (1,\infty)$ as defined in Proposition \ref{prop:stability-sobolev-DDSDE} and given by the relation
\begin{align*}
  \frac{1}{r'}=1-\frac{q'}{p},\quad \frac{1}{q}+\frac{1}{q'}=1.
\end{align*}
We deduce that mean field convergence holds along any sequence $\{Y^i\}$ such that $\PP$-a.s.
\begin{align*}
  d_{r'}(L^N(Y^{(N)}(\omega)), \mu^{Y^1})\to 0
\end{align*}
which includes the case of $\{Y^i\}$ i.i.d. distributed as $\mu^{\xi+W}\in \cP_{r'}(C_T)$.
\end{cor}

\begin{proof}
The proof follows the same line of arguments as the previous corollaries, this time based on Propositions \ref{prop:abstract-tanaka} and \ref{prop:stability-sobolev-DDSDE}.
\end{proof}

In particular Corollary \ref{cor:MFL-convol} holds in the case $\mu^{\xi+W}\in \cap_{n\in \NN} \cP_n(C_T)$, completing the proof of Theorem \ref{thm:intro-main}.

\appendix
\section{Approximation in metric spaces}\label{app:approximation}
Our main result from Section \ref{subsec:ContDrift} relies on approximating suitable drifts $B_t(\mu,x)$ by more regular ones. For this reason, we provide here an abstract result on approximation by Lipschitz functions in metric spaces $(E,d)$; although our main focus is given by $E=\RR^d\times \cP_p(\RR^d)$, the result is true in general.

Standards results of this kind can be found in \cite[Theorem 6.8]{heinonen2012lectures} and \cite[Section 6]{cobzacs2019lipschitz}; however, as they are not directly suited for our purpose, we present the necessary modifications below.
\begin{ass}\label{ass:TimeInHomLG}
Given a metric space $(E,d)$, we say that a measurable map $f :\RR_+\times E \rightarrow \RR^d$, has time inhomogeneous, linear growth if there exists a locally integrable function $h:\RR_+\rightarrow \RR_+$ and some (w.l.o.g. any) $z_0\in E$ such that
\begin{equation}
  |f_t(z)| \leq h_t\, (1+d(z,z_0)) \quad \forall\, t\geq 0, \, z\in E.
\end{equation}
\end{ass}
In order to prove Proposition \ref{prop:TimeInhomLipApprox} below, we first need the following lemma.
\begin{lem}\label{lem:approx lipschitz functions}
Let $g:E\to\RR$ be a bounded map, uniformly continuous on bounded sets. Define
\begin{equation}\label{eq:approxBaireLipschitz}
g^n(z) = \inf_{w\in E} \big\{g(w)+n d_E(z,w)\big\};
\end{equation}
the sequence $\{g^n\}_n$ has the following properties:
\begin{enumerate}[label=\roman*.]
\item \label{it:BddApproxLip} $g^n$ is $n$-Lipschitz;
\item \label{it:BddApproxMonotone} $g^n(z)\leq g^{n+1}(z) \leq g(z)$ and $|g^n(z)|\leq \| g\|_{C^0}$ for all $z\in E$ and $n\in \NN$;
\item \label{it:BddApproxConverge} $g^n\to g$ uniformly on bounded sets.
\end{enumerate}
\end{lem}

\begin{proof}
Let us set $K=\| g\|_{C^0}$; we can assume $K\neq 0$. Point \ref{it:BddApproxLip} and the first part of \ref{it:BddApproxMonotone} follow from \cite[Theorem 6.4.1]{cobzacs2019lipschitz}; by the fact that $g(w)+n d_E(z,w) \geq g(w)\geq -K$ we deduce $g^n(z)\geq -K$ and thus $|g^n(z)|\leq K$ for all $z\in E$ and $n\in\NN$. The proof of \ref{it:BddApproxConverge} is mostly the same as \cite[Theorem 6.4.1]{cobzacs2019lipschitz}, which however only deals with globally uniformly continuous functions; for this reason we give it explicitly.

Let us take $n$ big enough so that $2K/n< 1$ and fix $z_0\in E$, $R>0$; let $m_R$ denote the modulus of continuity of $g$ on $B_{R+1}(z_0)$, which exists by assumption. For any $w\in E$ such that $d(z,w)\geq 2K/n$ it holds that
\begin{align*}
g(z) + n d_E(z,z) = g(z) \leq K = 2K-K \leq g(w)+ n d_E(z,w)
\end{align*}
which implies that
\begin{align*}
g^n(z) = \inf \big\{g(w)+n d_E(z,w) : w\in E, d_E(w,z)\leq 2K/n\big\}.
\end{align*}
In particular if $z\in B_R(z_0)$, since $2K/n<1$, the infimum is taken over a subset of $B_{R+1}(z_0)$. Fix $\eps>0$, $z\in B_R(z_0)$; we can find $w_n$ s.t. $d_E(w_n,z)\leq 2K/n$ and $g(w_n) + n d_E(w_n,z) \leq g^n(z)+\eps$. Thus
\begin{align*}
g(w_n)-g^n(z) \leq g(w_n)-g^n(z) + n d_E(w_n,z) \leq \eps
\end{align*}
and
\begin{align*}
0 \leq g(z)-g^n(z) = g(z)\mp g(w_n) -g^n(z) \leq m_R \Big(\frac{2K}{n}\Big) +\eps \to \eps \quad\text{ as } n\to\infty.
\end{align*}
It follows that for all $n$ large enough,
\begin{align*}
\sup_{z\in B_R(z_0)} |g(z)-g^n(z)| \leq 2\eps;
\end{align*}
by the arbitrariness of $\eps>0$, uniform convergence on bounded sets follows.
\end{proof}
Applying Lemma \ref{lem:approx lipschitz functions} to time inhomogeneous functions gives the following result.
\begin{prop}\label{prop:TimeInhomLipApprox}
Let $f:[0,T]\times E\rightarrow \RR^d$ be as in Assumption \ref{ass:TimeInHomLG} and such that, for every $t\in [0,T]$, $f_t(\,\cdot\,)$ is uniformly continuous on bounded sets. Then there exists a sequence of functions $\{f^n\}_{n}:[0,T]\times E\rightarrow \RR^d$ such that:
\begin{enumerate}[label=\roman*.]
  \item For every $n \in \NN$ and $t \in [0,T]$, $f^n_t(\,\cdot\,)$ is globally Lipschitz with constant $C=C(n)>0$ independent of $t$.\label{it:LipBound}
  \item For every $n \in \NN$, $f^n_t(\,\cdot\,)$ has time inhomogeneous linear growth, in the sense of Assumption \ref{ass:TimeInHomLG}, with growth function $h_t$.\label{it:Growth}
  \item For every $t\in [0,T]$, $f^n_t(\,\cdot\,)\rightarrow f_t(\,\cdot\,)$ uniformly on bounded sets.\label{it:Converge}
\end{enumerate}
\end{prop}
\begin{proof}
Define the map
\begin{equation*}
  g_t(z) := \frac{f_t(z)}{(1+d_E(z,z_0))};
\end{equation*}
for every $t\in [0,T]$, by assumption $g_t(\,\cdot\,)$ is bounded and uniformly continuous on bounded sets, so we can apply Lemma \ref{lem:approx lipschitz functions}; this way we obtain a sequence $g^n_t(x)$, defined for fixed $t$ by equation \eqref{eq:approxBaireLipschitz}. Next, consider $\rho \in C^\infty_c(\RR;[0,1])$ such that $\rho(x)\equiv 1$ for $|x|\leq 1$ and $\rho(x)\equiv 0$ for $|x|\geq 2$, and set
\begin{equation*}
  f^n_t(z) := g^n_t(z)\left(1+ d_E(z,z_0) \right) \rho \left(\frac{d_E(z,z_0)}{n} \right).
\end{equation*}

We claim that $f^n$ has the desired properties. Indeed, since $|g^n_t(z)|\leq |g_t(z)|\leq h_t$, we also have $|f^n_t(z)|\leq |g^n_t(z)|(1+d_E(z,z_0))\leq h_t (1+d_E(z,z_0))$, so that \ref{it:Growth} is satisfied. Item \ref{it:LipBound} is clear from the construction of $f^n$, while the convergence of item \ref{it:Converge} follows from \ref{it:BddApproxConverge} of Lemma \ref{lem:approx lipschitz functions}.
\end{proof}
\section{Analytic lemmas and inequalities involving maximal functions}\label{app:maximal}
We start by presenting a useful weak-strong convergence lemma.
\begin{lem}\label{lem:weak-conv-L1}
Let $X^n$ be a sequence of $\RR^d$-valued r.v. defined on the same probability space such that $X^n\to X$ in probability; assume that $\mu^n:=\cL(X^n), \mu:= \cL(X)$ all belong to $L^1_x$ and $\mu^n\rightharpoonup \mu$ weakly in $L^1_x$. Finally, consider a sequence $b^n\in L^\infty_x$ such that $\| b^n\|_{L^\infty_x} \leq C$ and $b^n(x)\to b(x)$ for Lebesgue-a.e. $x$. Then it holds that
\[
\lim_{n\to \infty} \EE[|(b^n\ast \mu^n)(X^n)- (b\ast\mu)(X)|] = 0.
\]
\end{lem}
\begin{proof}
We split the estimate as follows:
\begin{align*}
  \EE[|(b^n\ast \mu^n)(X^n)- (b\ast \mu)(X)|]
  & \leq \EE[(|b^n-b|\ast \mu^n)(X^n)] + \EE[|b\ast (\mu^n-\mu)|(X^n)]\\
  & \quad + \EE[|(b\ast \mu)(X^n)-(b\ast \mu)(X)]\\
  & =: I_1^n+ I_2^n+ I_3^n.
\end{align*}

To handle $I^n_1$, we recall the following classical fact, which can be checked using Egorov's theorem: since $|b^n-b|$ is a uniformly bounded sequence converging Lebesgue a.e. to $0$ and $\mu^n$ is equiintegrable, $|b^n-b|\ast \mu^n$ converges pointwise to $0$; but then iterating the procedure to $f^n:=|b^n-b|\ast \mu^n$ in place of $|b^n-b|$, we can conclude that $I_1^n=\int f^n\, \dd \mu^n\to 0$ as well.

Convergence of $I_2^n$ follows from similar arguments.

Finally, since $b\in L^\infty_x$ and $\mu\in L^1_x$, $b\ast \mu\in C^0_b$ and so convergence of $I^n_3$ follows from the uniform bounds and the convergence $X^n\to X$ in probability. 
\end{proof}
The remainder of this appendix is devoted to recalling some classical and more recent properties of maximal functions and their use in estimates involving Wasserstein distances.

Maximal functions play a fundamental role in ODEs driven by Sobolev drifts, see for instance \cite{crippa2008estimates} for their application; let us recall some fundamental facts, which can be found in \cite{stein1970singular}.  

Given $b\in L^p(\RR^d)$, $p\in [1,\infty]$, its maximal function $M b$ is defined by
\[
M b(x) := \sup_{r>0} \frac{1}{\lambda_d\, r^d} \int_{B(x,r)} |b(y)|\,\dd y
\]
where $\lambda_d$ stands for the Lebesgue measure of $B(0,1)$ in $\RR^d$. It is well known that if $p\in (1,\infty]$, then $M f\in L^p(\RR^d)$ and
\[
\|Mb \|_p\leq c_{d,p} \|b\|_p
\]
for some constant $c_{d,p}>0$; similar definition and properties hold in the case of vector-valued drifts $b\in L^p(\RR^d;\RR^m)$ (in which case $c=c_{d,p,m}$).

If $b\in W^{1,p}(\RR^d;\RR^d)$, then there exists a Lebesgue-negligible set $N\subset\RR^d$ and a constant $c_d>0$ such that the Hajlasz inequality holds:
\begin{equation}\label{eq: maximal function1}
|b(x)-b(y)|\leq c_d\, |x-y|\,( M Db(x) + MDb(y))\quad \forall\, x,y\in \RR^d\setminus N.
\end{equation}

The above and similar inequalities allows one to control the map $\mu\mapsto b\ast \mu$ in Wasserstein spaces; these results are relevant for applications in Section \ref{sec:convol}.
\begin{lem}\label{lem: maximal function1}
Let $(p,q,r)\in (1,\infty)^3$ be such that $b\in W^{1,p}_x$, $\mu,\nu \in L^q_x$ and
\begin{align*}
  \frac{r}{p}+\frac{1}{q}\leq 1.
\end{align*}
Then there exists a constant $C=C(d,p,q,r)$ such that
\begin{equation*}
\| b\ast(\mu-\nu) \|_\infty \leq C \| b\|_{W^{1,p}_x}\, \big(\| \mu\|_{L^q_x}^{1/r}+\|\nu\|_{L^q_x}^{1/r}\big)\, d_{r'}(\mu,\nu)
\end{equation*}
\end{lem}

\begin{proof}
If $d_{r'}(\mu,\nu)=+\infty$ the inequality is trivially true (recall our definition of $d_{r'}(\mu,\nu)$, which is valid for any $\mu,\nu \in \cP(\RR^d)$), so we can assume $d_{r'}(\mu,\nu)<\infty$. Moreover by hypothesis $\mu,\nu\in L^{\tilde{q}}_x$ for any $\tilde{q}\in [1,q]$, therefore w.l.o.g. we can assume $r/p+1/q=1$.

Let $m\in \Pi(\mu,\nu)$ be any coupling of $(\mu,\nu)$ and let $N\subset\RR^d$ be as in \eqref{eq: maximal function1}; since $\mu,\nu$ are absolutely continuous w.r.t. Lebesgue, it holds that $m((N\times\RR^d) \cup (\RR^d\times N))=0$. 
Therefore applying \eqref{eq: maximal function1}, for any fixed $x\in \RR^d$, gives that
\begin{align*}
|b\ast(\mu-\nu)(x)|
& = \bigg| \int_{\RR^{2d}} [b(x-y)-b(x-z)]\, m(\dd y,\dd z)]\bigg|\\
& \lesssim \int_{((\RR^d\times N) \cup (N\times \RR^d))^c} |y-z|\, (M Db(x-y)+MDb(x-z))\, m(\dd y,\dd z) \\
& \leq \bigg(\int_{\RR^{2d}} | y-z|^{r'} \,m(\dd y,\dd z) \bigg)^{1/r'}\, \bigg(\int_{\RR^{2d}} | M Db(x-y) + M Db(x-z)|^r \,m(\dd y,\dd z)\bigg)^{1/r}.
\end{align*}
Moreover we have
\begin{align*}
\int_{\RR^d} |M Db(x-y)|^r \,m(\dd y,\dd z)
= \int_{\RR^d} |M Db(x-y)|^r \mu(\mathd y)
\leq \| M Db\|_{L^{rq'}}^r \|\mu \|_{L^q}
\lesssim \| Db\|_{L^{rq'}}^r \|\mu \|_{L^q};
\end{align*}
since $rq'=p$, applying a similar estimate to $M Db(x-z)$ we obtain that
\begin{equation*}
\sup_{x\in \RR^d} |b\ast(\mu-\nu)(x)| \lesssim \|b\|_{W^{1,p}}\, \big(\|\mu\|_{L^q}^{1/r}+\|\nu\|_{L^q}^{1/r}\big)\, \bigg(\int_{\RR^{2d}} | y-z|^{r'} \,m(\dd y,\dd z) \bigg)^{1/r'}.
\end{equation*}
Minimizing over all possible $m\in \Pi(\mu,\nu)$ we obtain the conclusion.
\end{proof}

Inequality \eqref{eq: maximal function1} is an example of a family of possibly asymmetric inequalities involving $M b$, see \cite{brue2021positive} for a general discussion.
It was shown in \cite[Lemma 5.1]{caravenna2021directional} that, under the stricter requirement that $b\in W^{1,p}(\RR^d;\RR^d)$ for $p>d$, there exists a function $g\in L^p(\RR^d)$ (which can be taken as $g=c_{d,p}\, (M |Db|^{\tilde{p}})^{1/\tilde{p}}$ for any $\tilde{p}\in (d,p)$) and a Lebesgue-negligible set $N\subset\RR^d$ such that $\| g\|_{L^p} \lesssim_{d,p} \| f\|_{W^{1,p}}$ and 
\begin{equation}\label{eq:maximal-function-onesided}
  |f(x)-f(y)|\leq g(x)\,|x-y| \quad \forall\, x,y\in \RR^d\setminus N.
\end{equation}
%

Contrary to Lemma \ref{lem: maximal function1}, which requires both $\mu,\nu\in L^q_x$, the one-sided inequality \eqref{eq:maximal-function-onesided} provides the following estimate.

\begin{lem}\label{lem:maximal-function2}
Let $(p,q,r)\in (1,\infty)^3$ be parameters satisfying
\begin{equation}\label{eq:maximal-function2-coeff}
  p>d, \quad \frac{r}{p}+\frac{1}{q}\leq 1.
\end{equation}
Then there exists a constant $C=C(d,p,q,r)$ such that, for any $b\in W^{1,p}_x$ and $\mu,\nu \in \cP$ with $\mu\in L^q_x$, it holds that
\begin{equation*}
| (b\ast\mu)(x)-(b\ast\nu)(y)| \leq C \| b\|_{W^{1,p}} (1+\| \mu\|_{L^q})\,\big(|x-y|+ d_{r'}(\mu,\nu)\big).
\end{equation*}
\end{lem}

\begin{proof}
As in the proof of Lemma \ref{lem: maximal function1}, we only need to consider the case $d_{r'}(\mu,\nu)<\infty$ and may take $b$ to be smooth and $r/p+1/q=1$ if needed.

Since $p>d$, by the Sobolev embedding $W^{1,p}\hookrightarrow C^0_b$, so $b\ast \mu$ and $b\ast\nu$ are well defined for any $\mu,\nu\in\cP(\RR^d)$; by Young's inequality
\[
|(b\ast \mu)(x)-(b\ast \mu)(y)|
\leq \| b\ast \mu\|_{W^{1,\infty}} |x-y|
\lesssim \| b\|_{W^{1,p}} (1+\| \mu\|_{L^q})|x-y|.
\]

Let $m\in \Pi(\mu,\nu)$, then by virtue of \eqref{eq:maximal-function-onesided} (as before the negligible set $N$ doesn't play any role), going through similar calculations to Lemma \ref{lem: maximal function1} we have
\begin{align*}
  |(b\ast\mu)(y)-(b\ast\nu)(y)|
  & \leq \bigg(\int_{\RR^d} |g|^{r} (y-z_1)\, \dd \mu(z_1)\bigg)^{1/r}\, \bigg( \int_{\RR^{2d}} |z_1-z_2|^{r'}\,\dd m(z_1,z_2) \bigg)^{1/r'}\\
  & \leq \| g\|_{L^{rq'}} \| \mu \|_{L^q}^{1/r} \, \bigg( \int_{\RR^{2d}} |z_1-z_2|^{r'}\,\dd m(z_1,z_2) \bigg)^{1/r'}\\
  & \lesssim \| b\|_{W^{1,p}} (1+\| \mu\|_{L^q})\,\bigg( \int_{\RR^{2d}} |z_1-z_2|^{r'}\,\dd m(z_1,z_2) \bigg)^{1/r'}
\end{align*}
where in the last passage we used the fact that $rq'=p$ and $\| g\|_{L^p} \lesssim \| b\|_{W^{1,p}}$. Combining the two estimates above and taking the infimum over all possible couplings $m\in\Pi(\mu,\nu)$ yields the conclusion.
\end{proof}

\bibliography{all}{}
\bibliographystyle{plain}

\end{document}